\documentclass[a4paper, 11pt]{amsart}

\usepackage[T1]{fontenc}

\usepackage{amssymb}
\usepackage[utf8]{inputenc}
\usepackage{hyperref}
\usepackage{bm}
\usepackage{comment}
\usepackage{xcolor}

\makeatletter
\@namedef{subjclassname@2020}{%
  \textup{2020} Mathematics Subject Classification}
\makeatother

\allowdisplaybreaks[2]

\def\supp{{\rm Supp}}

\newtheorem{theorem}{Theorem}[section]
\newtheorem{proposition}[theorem]{Proposition}
\newtheorem{lemma}[theorem]{Lemma}
\newtheorem{corollary}[theorem]{Corollary}
\theoremstyle{definition}
\newtheorem{remark}[theorem]{Remark}
\newtheorem{definition}[theorem]{Definition}
\newtheorem{example}[theorem]{Example}

\numberwithin{equation}{section}

\begin{document}

\title{$\mathbf{f}$-zpd algebras and a multilinear Nullstellensatz}

\author{Žan Bajuk}
\address{Faculty of Mathematics and Physics, University of Ljubljana, Slovenia}
\email{zanbajuk@gmail.com}

\author{Matej Bre\v{s}ar}
\address{Faculty of Mathematics and Physics, University of Ljubljana \&
Faculty of Natural Sciences and Mathematics, University of Maribor \& IMFM, Ljubljana, Slovenia}
\email{matej.bresar@fmf.uni-lj.si, matej.bresar@um.si}

\author{Pedro Fagundes}
\address{Department of Mathematics, State University of Campinas, S\'ergio Buarque de Holanda 651, 13083-859 Campinas, SP, Brazil}
\email{pedro.fagundes@ime.unicamp.br}

\author{Antonio Ioppolo}
\address{Dipartimento di Ingegneria e Scienze dell'Informazione e Matematica, Universit\`a degli Studi dell'Aquila, Via Votoio, 67100, L'Aquila, Italy}
\email{antonio.ioppolo@univaq.it}
\thanks{M.\ Brešar was 
 supported by Grant P1-0288, ARIS (Slovenian Research and Innovation Agency).  P.\ Fagundes was supported by Grant 2019/16994-1 and Grant 2022/05256-2, São Paulo Research Foundation (FAPESP). A.\ Ioppolo was  supported by GNSAGA of INdAM and by Premio Giovani Talenti Bicocca 2021}



\begin{abstract}
Let $f=f(x_1,\dots,x_m)$ be a multilinear polynomial over a field $F$.   An $F$-algebra $A$ is said to be $f$-zpd ($f$-zero product determined) if every $m$-linear
functional $\varphi \colon A^m\to F$ which preserves zeros of $f$ is of the form $\varphi(a_1,\dots,a_m)=\tau\left(f(a_1,\dots,a_m)\right)$ for some linear functional $\tau$ on $A$. 
We are primarily interested in the question whether the matrix algebra $M_d(F)$ is $f$-zpd. 
While the answer is negative in general, 
we provide several families of polynomials for which it is positive. We also consider a related problem on the form of a multilinear polynomial
$g=g(x_1,\dots,x_m)$
with the property that every zero
of $f$ in $M_d(F)^m$ is a zero of $g$. Under the assumption that $m < 2d-3$, we show that $g$ and $f$ are linearly dependent.
\end{abstract}

\maketitle

\section{Introduction}
 
Throughout the paper, by an algebra we  mean an associative, unital algebra over a fixed field $F$. 
By $F\langle x_1, x_2,\dots\rangle$ we denote the free algebra in indeterminates $x_1, x_2,\dots$, and by $M_d(A)$ we denote the algebra of all $d\times d$ matrices over the algebra $A$. We tacitly assume  that $d\ge 2$.

We say that an
algebra $A$ is   {\em zero product determined}, zpd for short,    if every   bilinear functional $\varphi \colon A\times A\to F$ with the property that
$\varphi(a,b)=0$ whenever
 $ab=0 $  is of the form $\varphi(a,b)=\tau(ab) $ for some linear functional $\tau$. 
  Replacing in this definition the role of the ordinary product $ab$ by the Lie product
  $[a,b]=ab-ba$ (resp.\ the Jordan product $a\circ b= ab +ba$), we speak about the {\em zero Lie product determined algebra}, zLpd for short  (resp.\  {\em zero Jordan product determined algebra}, zJpd for short). 
 These three classes of  algebras  are the subject  of the  book \cite{Bresar2021}.  Their study was mainly motivated by a variety of applications to different mathematical areas.
 
 In this paper, we initiate the study of a more general notion
which we define by replacing the role of the (ordinary, Lie, Jordan) product in the above definitions
 by the evaluation of  an arbitrary multilinear polynomial. This may be viewed as a part of the general program of understanding zeros and images of noncommutative polynomials.

Let us make this precise. Recall 
 first  that a 
 noncommutative polynomial
 $f=f(x_1, \dots, x_m)
 \in F\langle x_1,x_2,\dots\rangle$
 is said to be 
 multilinear if each indeterminate $x_i$, $i=1,\dots,  m$, appears exactly once in each   monomial of $f$. That is, $f$ can be written as
 $$f=\sum_{\sigma\in S_m}\alpha_\sigma x_{\sigma(1)}x_{\sigma(2)}\cdots x_{\sigma(m)} $$ with
 $\alpha_\sigma\in F$, where $S_m$ is the symmetric group on $m$ elements. We will consider $m$-linear functionals $\varphi \colon A^m\to F$ that {\em preserve  zeros of $f$} (i.e., $\varphi$ satisfies (\ref{mainhyphotesis}) below). A natural example of such a  functional is obtained by composing 
a linear functional and the evaluation of $f$  in $A^m$ (see (\ref{linearfunctional}) below), and we will be interested in algebras $A$ in which there are no other examples. Our definition thus reads as follows.

\begin{definition} \label{generalproblem}
Let $f=f(x_1,\dots,x_m)$ be a multilinear polynomial. An $F$-algebra $A$ is said to be {\em $f$-zero product determined},
{\em $f$-zpd} for short,  if for every
$m$-linear functional
$\varphi \colon A^m\to F$  with the property that for all $a_{1},\dots,a_{m} \in A$,
\begin{equation}\label{mainhyphotesis}
f(a_{1},\dots, a_{m})=0 \implies \varphi(a_{1},\dots, a_{m})=0,
\end{equation} there exists a linear functional $\tau \colon A \to F$ such that
\begin{equation}\label{linearfunctional}
\varphi(a_{1}, \dots, a_{m})=\tau\left(f(a_{1}, \dots, a_{m})\right)  
\end{equation}
for all $a_1,\dots,a_m\in A$.
\end{definition}

Of course, 
$x_1x_2$-zpd is the same as zpd, 
$[x_1,x_2]$-zpd 
is the same as zLpd,
and 
$x_1\circ x_2$-zpd 
is the same as zJpd.  Let us first state a result which follows by combining several theorems from \cite{Bresar2021}.

\begin{theorem}\label{zLpdzJpd} \cite[Theorems 2.15, 3.10, and 3.15]{Bresar2021}
 Any  algebra $A$ generated by idempotents is  zpd and, provided that ${\rm char}(F)\ne 2$, also zJpd. Further, if $A$ is any  zLpd algebra, then  the matrix algebra $M_d(A)$, $d\ge 2$, is also zLpd. Consequently, the algebra  $M_d(F)$,  $d\ge 2$, is zpd, zLpd, and zJpd.
\end{theorem}

In other words,
$M_d(F)$ is $f$-zpd if $f$ is any of the polynomials $x_1x_2$, $[x_1,x_2]$, and $x_1\circ x_2$. The first question one might ask is whether  
$M_d(F)$ is $f$-zpd for every multilinear polynomial $f$. However, in Section \ref{s2}
we show that the answer is negative (Proposition \ref{matricesnotfzpd}). The example we give is based on an arbitrary central polynomial. As such polynomials exist in abundance, the question of whether or not $M_d(F)$ is $f$-zpd   is more delicate than we, admittedly, initially expected.
In Section \ref{s3}, we present three approaches yielding positive answers  for some polynomials $f$. The first one is based on constructing new polynomials for which the answer is positive from old. The main result, Theorem \ref{invariantunderendomorphisms}, %involves conditions that are related to the Lvov-Kaplansky conjecture \cite{R}. It  
concerns general algebras and is rather technical, but has some easily stated corollaries.  For example, under a mild assumption on char$(F)$, $M_d(F)$ is $f$-zpd for every multilinear Lie monomial $f$ (Corollary \ref{multliemonomialszpd}) as well as for every  multilinear Jordan monomial $f$ (Corollary \ref{l}).  The second approach is applicable to algebras that are generated by idempotents, of which $M_d(F)$ is a particular case. We show that such algebras are $f$-zpd for certain polynomials $f$ arising from cyclic permutations (Theorem \ref{Zanpolynomial}),  thereby extending the first assertion of  
Theorem \ref{zLpdzJpd}. The third approach is direct, computational, and is used only for the generalized commutator
$f=x_1x_2x_3 - x_3x_2x_1$ (Theorem \ref{existenceoftau}), which we believe to be of special interest but is not covered by the first two more sophisticated approaches.

In Section \ref{s4}, we consider an apparently independent, but clearly similar (and  somewhat easier) problem of relating two multilinear polynomials 
$f=f(x_1,\dots,x_m)$ and $g=g(x_1,\dots,x_m)$ with the property that every zero of $f$ in $M_d(F)^m$ is also a zero of $g$. This condition is a special case of the one from Amitsur's Nullstellensatz \cite{Am} which, however, is a direct generalization of Hilbert's Nullstellensatz  and so the conclusion involves a power of the polynomial $g$. Since $M_d(F)$ contains nonzero nilpotents,  Amitsur's Nullstellensatz gives only a  necessary condition for the zero set inclusion.  The problem that we addressed is to provide, in the aforementioned special case of multilinear polynomials $f$ and $g$, a condition that is both necessary and sufficient. A natural candidate is the condition 
 that $g$ is the sum of a scalar multiple of $f$ and a polynomial identity.  However, essentially the same example as the one showing that $M_d(F)$ is not always $f$-zpd 
also shows that this is not always the case; moreover,
the polynomials $f$ and $g$ from our example have the same zero sets 
(Proposition \ref{propertyofg}). To obtain the desired conclusion, we need the additional assumption that $m < 2d-3$. Since  polynomial identities cannot occur in this case,  the result is simply that $f$ and $g$ are linearly dependent (Theorem
\ref{Null}).

\section{A counterexample for the general case}\label{s2}

We start by recalling some standard definitions and facts from the theory of polynomial identities. The reader can find them in any standard book on the subject. 

We say that a polynomial $p=p(x_1,\dots,x_n)\in F\langle x_1,x_2,\dots\rangle $  is a {\em polynomial identity}
of an $F$-algebra $A$ if
$p(a_1,\dots,a_n)=0$ for all $a_1,\dots,a_n\in A$. The matrix algebra $M_d(F)$ has no polynomial identities of degree less than  $2d$, and satisfies a multilinear polynomial identity of degree $2d$, namely the standard polynomial identity $$s_{2d}= \sum_{\sigma\in S_{2d}} {\rm sgn}(\sigma) x_{\sigma(1)}x_{\sigma(2)}\cdots x_{\sigma(2d)}.$$ A polynomial $c=c(x_1,\dots,x_n)\in F\langle x_1,x_2,\dots\rangle $ with zero constant term is said to be a {\em central polynomial} of $A$ 
if 
$c(a_1,\dots,a_n)$ lies in the center of $A$ for all $a_1,\dots,a_n\in A$, but $c$ is not a polynomial identity of $A$. Central polynomials of $M_d(F)$ exist for every $d\ge 2$.
%(see \cite{Formanek, Razmyslov}). 
The question of what is the minimal degree $m_d$ of a central polynomial of $M_d(F)$ is open, but it is known that $2d\le m_d\le d^2$. 

Until the rest of this section, assume that the field $F$ has more than $3$ elements. Pick
$\alpha,\beta \in F\setminus \{0,-1\}$ with $\alpha\ne \beta$. 
Fix $d\ge 2$ and  a   multilinear central polynomial
 $c=c(x_1,\dots,x_n)$ of $M_d(F)$, and  define  multilinear polynomials $h_1,h_2,f,g$ of degree $m=n+1$  by
\begin{align*} h_1 &= c(x_1,\dots,x_{m-2},x_{m-1})x_{m},\\
h_2&= c(x_1,\dots,x_{m-2},x_{m})x_{m-1},\\
f &=h_1+ \alpha h_2,\\
g&=h_1+\beta h_2.\end{align*}

\begin{example}\label{e5}
If $d=2$, then we can take
\begin{equation*}\label{xx}c=[x_1,x_2][x_3,x_4] + [x_3,x_4] [x_1,x_2] \end{equation*}
which is a  central polynomial of  minimal degree.
Then $h_1,h_2,f,g$ are of degree  $m=5$. For example,
\begin{align*}
f= [x_1,x_2][x_3,x_4]x_5 +& [x_3,x_4] [x_1,x_2]x_5 \\+
 &\alpha [x_1,x_2][x_3,x_5]x_4 + \alpha [x_3,x_5] [x_1,x_2]x_4.
 \end{align*}
\end{example}

The next proposition is related to the multilinear Nullstellensatz from Section \ref{s4}.

\begin{proposition}
    \label{propertyofg}
    Let $a_1,\dots,a_{m}\in M_d(F)$. The following conditions are equivalent:
    \begin{enumerate}
        \item[(i)] $f(a_1,\dots,a_{m})=0$.
         \item[(ii)] $g(a_1,\dots,a_{m})=0$.
        
        \item[(iii)] $h_1(a_1,\dots,a_{m})=h_2(a_1,\dots,a_{m})=0$.
      
    \end{enumerate}
   In particular,
   $f$ and $g$ have the same zero sets. 
However, $g$ is not the sum of a scalar multiple of $f$ and a polynomial identity.
\end{proposition}
\begin{proof} % We first prove that (i) implies (iii).
Suppose 
$f(a_1,\dots,a_{m})=0$ but
$h_1(a_1,\dots,a_{m})\ne 0$. The latter implies
$c(a_1,\dots,a_{m-1})\ne 0$ which
together with $f(a_1,\dots,a_{m})=0$ shows that 
$a_{m}=\lambda a_{m-1}$ 
for some $\lambda\in F$.  Hence,
\begin{align*}
&(1+\alpha)h_1(a_1,\dots,a_{m})\\ =&\lambda(1+\alpha) c(a_1,\dots,a_{m-1})a_{m-1}\\
=& c(a_1,\dots,a_{m-2},a_{m-1})a_{m}+\alpha c(a_1,\dots,a_{m-2},a_{m})a_{m-1}\\=&
f(a_1,\dots,a_{m}) =  0.\end{align*}
As $\alpha\ne -1$, this  contradicts our assumption.  We have thereby shown that (i) implies (iii). Since (iii) trivially implies (i), these two conditions are equivalent. Similarly we see that (ii) and (iii) are equivalent.

It is easy to check that
$g$ is not the sum of a scalar multiple of $f$ and a polynomial identity.
\end{proof}

The second proposition provides an evidence for the nontriviality of the results of Section \ref{s3}.

\begin{proposition} \label{matricesnotfzpd}
$M_d(F)$ is not $f$-zpd.
\end{proposition}
\begin{proof}
Pick
$u_1,\dots, u_{m-1}\in M_d(F)$ such that $c(u_1,\dots,u_{m-1})\ne 0$. Hence, $$\rho(a) = c(u_1,\dots,u_{m-2},a)$$ is a nonzero linear functional on $M_d(F)$ (here we  identified scalars with scalar multiples of the identity). Let $\omega$ be any linear functional on $M_d(F)$  that is linearly independent to
$\rho$. Define $\varphi \colon M_d(F)^{m}\to F$ by
$$\varphi(a_1,\dots,a_{m}) = c(a_1,\dots,a_{m-1})\omega(a_{m}).$$
Observe that the implication (i)$\implies$(iii) from Proposition \ref{propertyofg} shows that for all $a_1,\dots,a_{m}\in M_d(F)$,
$$f(a_1,\dots,a_{m})=0\implies \varphi(a_1,\dots,a_{m})=0. $$
Suppose $M_d(F)$ was $f$-zpd. Then 
there would exist a linear functional
$\tau$ on $M_d(F)$ such that 
$$\varphi(a_1,\dots,a_{m}) = \tau\left(f(a_1,\dots,a_{m})\right)$$
for all $a_1,\dots,a_{m}\in M_d(F)$.
That is,
\begin{align*}
&c(a_1,\dots,a_{m-2},a_{m-1})\omega(a_{m}) \\=& c(a_1,\dots,a_{m-2}, a_{m-1})\tau(a_{m}) + \alpha c(a_1,\dots,a_{m-2},a_{m})\tau(a_{m-1}).
\end{align*}
Taking $u_i$ for $a_i$, $i=1,\dots,m-2,$ and writing $a$ for $a_{m-1}$ and $b$ for $a_{m}$, we thus have 
$$ \rho(a)\omega(b) = \rho(a)\tau(b) + \alpha \tau(a)\rho(b) $$
for all $a,b\in M_d(F)$. Picking any $a\notin \ker\,\rho$ we see that 
$\tau =\omega +  \lambda \rho $
for some $\lambda\in F$. Consequently,
$$ ((1+\alpha)\lambda\rho(a)  +  \omega(a))\rho(b) =0,$$
%for all $a,b\in M_d(F)$, 
which contradicts the linear independence of $\rho$ and $\omega$.    
\end{proof}

\section{$f$-zpd algebras}\label{s3}

We just saw that the algebra $M_d(F)$ is not $f$-zpd for every multilinear polynomial $f$. The goal of this section is  to provide three different approaches yielding positive results, that  is,  results stating that some algebras, including $M_d(F)$,  are $f$-zpd for some  special polynomials $f$.

We start with a few simple observations. Throughout, $f$ will 
 be a multilinear polynomial of degree $m$.

\begin{remark} \label{zpdproportionalpolynomials}
Let $\alpha \in F$ be a nonzero scalar. Then $A$ is $f$-zpd if and only if $A$ is $\alpha f$-zpd.
\end{remark}

\begin{remark} \label{problemforPIs}
If $f$ is a polynomial identity of $A$, then $A$ is $f$-zpd. 
\end{remark}

The next three lemmas are straightforward generalizations of  standard results on zpd-algebras \cite{Bresar2021}.

\begin{lemma} \label{equivalentconditiontau}
Let $\alpha = f(1, \ldots, 1) \in F$ be nonzero. Then $A$ is $f$-zpd if and only if every $m$-linear functional $\varphi \colon A^m\to F$ that preserves  zeros of $f$ satisfies 
\begin{eqnarray*} \label{zancondition}
\varphi(a_1,\ldots, a_m) = \alpha^{-1} \varphi\left(f(a_1,\ldots,a_m),1,\ldots,1\right)
\end{eqnarray*}  for all $a_{1},\dots,a_{m}\in A$.
\end{lemma}
\begin{proof}
If $A$ is 
 $f$-zpd and $\tau$ is the linear functional from the definition, then 
\begin{equation*}
\begin{split}
\varphi(f(a_1,\ldots,a_m), 1, \ldots, 1) &= \tau(f(f(a_1,\ldots,a_m),1,\ldots,1)) \\ 
& = \alpha \tau(f(a_1,\ldots,a_m)).
\end{split}\end{equation*}
The reverse implication is obvious.
\end{proof}

\begin{lemma} \label{equivalentconditiontau2}
An $F$-algebra $A$ is $f$-zpd if and only if
every $m$-linear functional $\varphi \colon A^m\to F$ that preserves  zeros of $f$ satisfies the following condition:  for all $N\ge 1$ and all $a_{1}^{(t)},\dots,a_{m}^{(t)}\in A$, $t=1,\dots,N$,
\begin{eqnarray} \label{equivalentlinear}
\sum_{t=1}^Nf\bigl(a_{1}^{(t)},\dots, a_{m}^{(t)}\bigr)=0 \implies \sum_{t=1}^N\varphi\bigl(a_{1}^{(t)},\dots, a_{m}^{(t)}\bigr)=0.
\end{eqnarray}
\end{lemma}
\begin{proof}
The "only if" part is clear. To prove the "if" part, denote by $A_0$ the linear span of all
$f(a_1,\dots,a_m)$, $a_i\in A$, 
and observe that (\ref{equivalentlinear})
implies that $\tau_{0} \colon A_0 \to F$,
$$
\tau_{0} \left( \sum_{t=1}^N f \left( a_1^{(t)}, \ldots, a_m^{(t)} \right) \right) = \sum_{t=1}^N \varphi \left(a_1^{(t)}, \ldots, a_m^{(t)} \right)
$$
is a well defined linear functional on $A_0$. Letting 
 $\tau \colon A^{m} \to F$ to be  any linear extension of $\tau_{0}$, we thus have  $\varphi(a_1,\dots,a_m)=\tau\left(f(a_1,\dots,a_m)\right)$ for all $a_1,\dots,a_m\in A$.
\end{proof}

\begin{lemma} \label{l35}
Let $A$ be an $f$-zpd algebra and let 
 $X$ be a vector space over $F$. If an
 $m$-linear map $\Phi \colon A^m\to X$ preserves zeros of $f$, i.e., 
 for all $a_1,\dots,a_m\in A$, 
 $$ f(a_1,\dots,a_m)=0\implies  \Phi(a_1,\dots,a_m)=0,$$ then there exists a linear map $T \colon A\to X$ such that 
 $$\Phi(a_1,\dots,a_m)=T\left(f(a_1,\dots,a_m)\right) $$
 for all $a_1,\dots,a_m\in A$. 
\end{lemma}
\begin{proof}
If $X=F$ then this is true by the definition of an $f$-zpd algebra. The general case can be easily reduced to this one. Indeed, take a linear functional $\omega$ on $X$ and observe that the composition $\omega\circ \Phi$ is an $m$-linear functional preserving zeros of $f$. We may therefore use Lemma \ref{equivalentconditiontau2} to conclude that
 for all $a_{1}^{(t)},\dots,a_{m}^{(t)}\in A$,
$$\sum_{t=1}^Nf\bigl(a_{1}^{(t)},\dots, a_{m}^{(t)}\bigr)=0 \implies $$ 
$$\omega\left(\sum_{t=1}^N\Phi\bigl(a_{1}^{(t)},\dots, a_{m}^{(t)}\bigr)\right)=  \sum_{t=1}^N(\omega\circ \Phi)\bigl(a_{1}^{(t)},\dots, a_{m}^{(t)}\bigr)=0.$$ 
%\begin{align*} 
%\sum_{t=1}^Nf(a_{1}^{(t)},\dots, a_{m}^{(t)})=0 \implies &  \omega\left(\sum_{t=1}^N\Phi(a_{1}^{(t)},\dots, a_{m}^{(t)})\right)\\&=  \sum_{t=1}^N(\omega\circ \Phi)(a_{1}^{(t)},\dots, a_{m}^{(t)})=0.
%\end{align*}
Since $\omega$ is an arbitrary linear functional on $X$, it follows that $\Phi$ satisfies 
\begin{eqnarray*} 
\sum_{t=1}^N f\bigl(a_{1}^{(t)},\dots, a_{m}^{(t)}\bigr)=0 \implies \sum_{t=1}^N\Phi\bigl(a_{1}^{(t)},\dots, a_{m}^{(t)}\bigr)=0.
\end{eqnarray*}
We can now repeat the argument from the proof of Lemma  \ref{equivalentconditiontau2}, that is, we define  the linear map
$T_0 \colon A_0\to X$ by
$$
T_{0} \left( \sum_{t=1}^N f \left( a_1^{(t)}, \ldots, a_m^{(t)} \right) \right) = \sum_{t=1}^N \Phi \bigl(a_1^{(t)}, \ldots, a_m^{(t)}\bigr)
$$
and extend it to a linear map $T \colon A\to X$.
\end{proof}

This last lemma should be  crucial for possible applications of 
the concept of an $f$-zpd algebra. However, we will not discuss them in this paper.

The remaining of the section is divided into three subsections.

\subsection{Compositions of polynomials}

To state our first theorem  in this section, we need the following generalization of Definition \ref{generalproblem}.

\begin{definition}
Let $f = f(x_1,\ldots,x_m)\in F\langle x_1,x_2,\dots \rangle$ be a multilinear polynomial, let $A$ be an $F$-algebra, and let   $B_1,\dots, B_m$ be vector subspaces of $A$. We say that {\em  the set $ B_1 \times \cdots \times B_m$ is  $f$-zpd} if for every $m$-linear functional $\varphi \colon B_1 \times \cdots \times B_m \to F$ with the property that for all
$a_i\in B_i$, $i=1,\dots,m$,  $$f(a_1,\ldots, a_m) = 0 \implies \varphi(a_1,\ldots, a_m)=0,$$  there exists a linear functional $\tau$ on $A$ such that 
$$
\varphi(a_1,\ldots, a_m) = \tau\left(f(a_1,\ldots,a_m)\right)
$$
for all $a_i\in B_i$, $i=1,\dots,m$.
\end{definition}

Of course, $A^m=A\times \cdots\times A$ is $f$-zpd is the same thing as $A$ is $f$-zpd.

Before stating the next theorem, we recall the 
 {\em L'vov-Kaplansky conjecture} which states that the image of any multilinear polynomial $f$ in $M_d(F)$ is a vector space (provided that $F$ is infinite). More specifically, if $f$ is not a polynomial identity or a central polynomial, then $$f(M_d(F))=\{f(a_1,\dots,a_m)\,|\,a_i\in M_d(F)\}$$ is either $M_d(F)$ or
 ${\rm sl}_d(F)$, i.e., the Lie subalgebra of $M_d(F)$ consisting of matrices  with zero trace (note that $f(M_d(F))$ is used as an abbreviation for 
 $f(M_d(F)\times \cdots\times M_d(F) )$). We refer the reader to the survey paper \cite{R} on images of polynomials, which is primarily devoted to this conjecture. Although only some partial solutions are known at present, the fact that this conjecture exists  
indicates that the 
 assumption (a) from the following theorem  is not artificial.
 The theorem actually concerns an arbitrary algebra $A$, but of course we are primarily interested in the case where $A=M_d(F)$.

\begin{theorem} \label{invariantunderendomorphisms}
Let $A$ be an $F$-algebra and let $k\ge 1$. For each $i=1,\dots, k$, let $f_{i} \left( x_{i_1},\dots,x_{i_{m_{i}}} \right)\in F\langle x_1,x_2,\dots \rangle$ be a multilinear polynomial in $m_i$ variables 
and let $B^{(i)}_{1}, \ldots, B^{(i)}_{{m_i}}$ be vector subspaces of $A$. Set $$C_i=B^{(i)}_{1} \times \cdots \times B^{(i)}_{{m_i}}.$$
Further, let $f_{0}\in F\langle x_1,x_2,\dots \rangle$ be a multilinear polynomial in $k$ variables, and let
$$
f=f_{0} \left( f_{1},\ldots,f_{k} \right).
$$
Suppose that the following three
conditions are satisfied:
\begin{enumerate}
    \item[(a)] For each $i=1,\dots,k$, $f_{i}(C_i)$
 is a vector subspace of $A$. 
\item[(b)] For each $i=1,\dots,k$, the set $C_i$ is $f_i$-zpd. 
\item[(c)] 
The set
$
f_1(C_1) \times \cdots \times f_k(C_k)
$ is
$f_{0}$-zpd.
\end{enumerate}
Then  the set $C_1\times\cdots\times C_k$ is $f$-zpd.
\end{theorem}

\begin{proof}
For each $i=1,\dots,k$,
we set   $$D_{i}= C_i \times C_{i+1} \times \cdots \times C_k.$$ Our goal is to prove that $D_1$ is $f$-zpd. 
Thus, take 
a multilinear functional $\varphi \colon D_1 \to F$  that preserves zeros of $f$, that is, for all $ a^{(i)}_j \in B^{(i)}_j$,
$$
f \left( a^{(1)}_{1}, \ldots, a^{(k)}_{m_{k}} \right)=0 
\implies
\varphi \left(a^{(1)}_{1}, \ldots, a^{(k)}_{m_{k}} \right) = 0.
$$
Define the multilinear functional $\varphi_{a_1^{(2)}, \ldots, a^{(k)}_{m_{k}}} \colon C_1 \to F$  by 
$$
\varphi_{a^{(2)}_{1}, \ldots, a^{(k)}_{m_{k}}}
\left( a^{(1)}_{1}, \ldots, a^{(1)}_{m_{1}} \right) =
\varphi \left( a^{(1)}_{1}, \ldots, a^{(1)}_{m_{1}}, a^{(2)}_{1}, \ldots, 
a^{(k)}_{m_{k}} \right).
$$
Clearly $\varphi_{a^{(2)}_{1}, \ldots, a^{(k)}_{m_{k}}}$ preserves zeros of $f_{1}$ on $C_1$. Since $C_1$ is $f_1$-zpd, there is a linear functional 
$\tau_{a^{(2)}_{1}, \ldots, a^{(k)}_{m_{k}}} \colon A\to F$ such that
$$
\varphi_{a^{(2)}_{1}, \ldots, a^{(k)}_{m_{k}}} \left( a^{(1)}_{1}, \ldots, a^{(1)}_{m_{1}} \right) = \tau_{a^{(2)}_{1}, \ldots, a^{(k)}_{m_{k}}} \left( f_{1} \left(a^{(1)}_{1}, \ldots, a^{(1)}_{m_{1}} \right) \right).
$$
Now define  $\Phi_{1} \colon f_1(C_1) \times D_2 \to F$    by 
\[
\Phi_{1} \left(a_{1}, a^{(2)}_{1}, \ldots, a^{(2)}_{m_{2}},\ldots, a^{(k)}_{1}, \ldots, a^{(k)}_{m_{k}} \right) = \tau_{a^{(2)}_{1}, \ldots, a^{(k)}_{m_{k}}}(a_{1}).
\]
The linearity of $\Phi_{1}$ in the first argument is obvious. Let us prove the linearity of $\Phi_{1}$ in the second argument. It is enough to show that 
\begin{equation}
\label{ee}
\tau_{a^{(2)}_{1} + b^{(2)}_{1}, a^{(2)}_{2},\ldots, a^{(m)}_{m_{k}}} = 
\tau_{a^{(2)}_{1}, a^{(2)}_{2},\ldots, a^{(m)}_{m_{k}}} +
\tau_{b^{(2)}_{1}, a^{(2)}_{2},\ldots, a^{(m)}_{m_{k}}}
\end{equation}
and
\begin{equation}
    \label{el}
\tau_{\lambda a^{(2)}_{1} , a^{(2)}_{2},\ldots, a^{(m)}_{m_{k}}} = 
\lambda \tau_{a^{(2)}_{1}, a^{(2)}_{2},\ldots, a^{(m)}_{m_{k}}}
\end{equation}
where $\lambda\in F$.
Take $a_{1} \in f(C_1)$ and write $a_{1}= f_{1} \left(
a^{(1)}_{1}, \ldots, a^{(1)}_{m_{1}} \right).$ Then,
\begin{eqnarray*}
\begin{aligned} 
\tau_{a^{(2)}_{1} + b^{(2)}_{1}, a^{(2)}_{2},\ldots, a^{(m)}_{m_{k}}} \left(a_{1} \right) &= 
\tau_{a^{(2)}_{1} + b^{(2)}_{1}, a^{(2)}_{2},\ldots, a^{(m)}_{m_{k}}} \left(f_{1} 
\left(a^{(1)}_{1}, \ldots, a^{(1)}_{m_{1}} \right) \right) \\
& =\varphi_{a^{(2)}_{1} + b^{(2)}_{1}, a^{(2)}_{2},\ldots, a^{(m)}_{m_{k}}}
\left( a^{(1)}_{1},\ldots, a^{(1)}_{m_{1}} \right)\\
&=\varphi \left( a^{(1)}_{1}, \ldots, a^{(1)}_{m_{1}}, a^{(2)}_{1}+ b^{(2)}_{1}, a^{(2)}_{2}, \ldots, a^{(k)}_{m_{k}} \right)\\
&=\varphi \left( a^{(1)}_{1},\ldots, a^{(1)}_{m_{1}}, a^{(2)}_{1}, a^{(2)}_{2},\ldots, a^{(k)}_{m_{k}} \right) \\ 
 & \ \ \ + \varphi \left( a^{(1)}_{1},\ldots, a^{(1)}_{m_{1}}, b^{(2)}_{1}, a^{(2)}_{2}, \ldots, a^{(k)}_{m_{k}} \right) \\
&=\varphi_{a^{(2)}_{1}, a^{(2)}_{2}, \ldots, a^{(k)}_{m_{k}}} 
\left( a^{(1)}_{1},\ldots, a^{(1)}_{m_{1}} \right) \\ 
 & \ \ \ +
\varphi_{b^{(2)}_{1}, a^{(2)}_{2}, \ldots, a^{(k)}_{m_{k}}} 
\left( a^{(1)}_{1},\ldots, a^{(1)}_{m_{1}} \right)\\
&= \tau_{a^{(2)}_{1}, a^{(2)}_{2}, \ldots, a^{(k)}_{m_{k}}}
\left( f_{1} \left( a^{(1)}_{1},\ldots, a^{(1)}_{m_{1}} \right) \right) \\ 
 & \ \ \ +
\tau_{b^{(2)}_{1}, a^{(2)}_{2}, \ldots, a^{(k)}_{m_{k}}}
\left( f_{1} \left( a^{(1)}_{1},\ldots, a^{(1)}_{m_{1}} \right) \right) \\
&=\left(
\tau_{a^{(2)}_{1}, a^{(2)}_{2}, \ldots, a^{(k)}_{m_{k}}} +
\tau_{b^{(2)}_{1}, a^{(2)}_{2}, \ldots, a^{(k)}_{m_{k}}} \right)(a_{1}),
\end{aligned}
\end{eqnarray*}
which proves (\ref{ee}). The proof of (\ref{el}) is similar. 
Analogously we see that $\Phi_1$ is linear in other arguments, so 
 $\Phi_{1}$ is a multilinear functional. Moreover, 
\begin{eqnarray*}
\begin{aligned} 
& \varphi \left( a^{(1)}_{1}, \ldots, a^{(1)}_{m_{1}}, a^{(2)}_{1}, \ldots, a^{(2)}_{m_{2}},\ldots, a^{(k)}_{1}, \ldots, a^{(k)}_{m_{k}} \right) \\
& \ \ \ \ \ \ \ \ \  =  \Phi_{1} \left(f_{1} \left(a^{(1)}_{1}, \ldots, a^{(1)}_{m_{1}} \right), a^{(2)}_{1}, \ldots, a^{(2)}_{m_{2}},\ldots, a^{(k)}_{1}, \ldots, a^{(k)}_{m_{k}} \right)
\end{aligned}
\end{eqnarray*}
for all $a^{(i)}_{j} \in B^{(i)}_{j}$.
From this we see that the multilinear functional $$\varphi_{a_{1}, a^{(3)}_{1}, \ldots, a^{(k)}_{m_{k}}} \colon C_2 \to F,$$ given   by
\[
\varphi_{a_{1}, a^{(3)}_{1}, \ldots, a^{(k)}_{m_{k}}}
\left( a^{(2)}_{1}, \ldots, a^{(2)}_{m_{2}} \right) = \Phi_{1}
\left( 
a_{1}, a^{(2)}_{1}, \ldots, a^{(2)}_{m_{2}}, a^{(3)}_{1}, \ldots, a^{(k)}_{m_{k}} \right),
\]
 preserves zeros of $f_{2}$ on $C_2$. As $C_2$ is $f_2$-zpd, there exists a linear functional $$\tau_{a_{1}, a^{(3)}_{1}, \ldots, a^{(k)}_{m_{k}}} \colon A\to F$$ such that 
$$
\varphi_{a_{1}, a^{(3)}_{1}, \ldots, a^{(k)}_{m_{k}}}
\left( a^{(2)}_{1}, \ldots, a^{(2)}_{m_{2}} \right) =
\tau_{a_{1}, a^{(3)}_{1}, \ldots, a^{(k)}_{m_{k}}}
\left(f_{2} \left( a^{(2)}_{1}, \ldots, a^{(2)}_{m_{2}} \right) \right).
$$
Next we  define $$\Phi_{2} \colon f_1(C_1) \times f_2(C_2) \times D_3 \to F$$
by
\[
\Phi_{2} \left( a_{1}, a_{2}, a^{(3)}_{1}, \ldots, a^{(3)}_{m_{3}}, \ldots, a^{(k)}_{1}, \ldots, a^{(k)}_{m_{k}} \right)=
\tau_{a_{1}, a^{(3)}_{1},\ldots, a^{(k)}_{m_{k}}} \left(a_{2} \right).
\]
We  show that $\Phi_{2}$ is multilinear in a similar fashion as we  showed that $\Phi_1$ is multilinear.
 Moreover, we have
\begin{eqnarray*}
    \begin{aligned}    
&\varphi \left( 
a^{(1)}_{1},\ldots, a^{(k)}_{m_{k}}\right)
\\ 
=& \Phi_{1} \left( f_{1} \left( a^{(1)}_{1},\ldots, a^{(1)}_{m_{1}} \right), a^{(2)}_{1}, \ldots, a^{(2)}_{m_{2}}, \ldots, a^{(k)}_{1}, \ldots, a^{(k)}_{m_{k}} \right) \\
=&\varphi_{f_{1} ( a^{(1)}_{1},\ldots, a^{(1)}_{m_{1}}), a^{(3)}_{1},\ldots, a^{(k)}_{m_{k}}} \left( a^{(2)}_{1}, \ldots, a^{(2)}_{m_{2}} \right) \\
=&\tau_{f_{1} ( a^{(1)}_{1},\ldots, a^{(1)}_{m_{1}}), a^{(3)}_{1},\ldots, a^{(k)}_{m_{k}}} \left( f_2 \left( a^{(2)}_{1}, \ldots, a^{(2)}_{m_{2}} \right) \right) \\
=&\Phi_{2} \left(
f_1 \left( a^{(1)}_{1}, \ldots, a^{(1)}_{m_{1}} \right), 
f_2 \left( a^{(2)}_{1}, \ldots, a^{(2)}_{m_{2}} \right), 
a^{(3)}_{1}, \ldots, a^{(k)}_{m_{k}}
\right).
\end{aligned}
\end{eqnarray*}
    
Repeating this process, we obtain the existence of a multilinear functional 
$$
\Phi \colon f_1(C_1) \times f_2(C_2) \times\cdots\times f_k(C_k) \to F
$$ 
satisfying 
\[
\varphi \left(
a^{(1)}_{1},\ldots,  a^{(k)}_{m_{k}} \right)
=\Phi \left(
f_{1} \left( a^{(1)}_{1},\ldots, a^{(1)}_{m_{1}} \right),
\ldots,
f_{k} \left( a^{(k)}_{1},\ldots, a^{(k)}_{m_{k}} \right) \right)
\]
 for all $a^{(i)}_{j} \in B^{(i)}_{j}$.
Suppose $f_{0}( a_{1},\ldots, a_{k})=0$ for some $a_i \in f(C_i)$. Writing $a_{i}=f_{i} \left( a^{(i)}_{1},\ldots, a^{(i)}_{m_{i}} \right)$, we then have
$$
f \left(
a^{(1)}_{1},\ldots, a^{(1)}_{m_{1}}, \ldots, a^{(k)}_{1},\ldots, a^{(k)}_{m_{k}} \right)=0. 
$$
Hence $$\varphi \left(
a^{(1)}_{1},\ldots, a^{(1)}_{m_{1}}, \ldots, a^{(k)}_{1},\ldots, a^{(k)}_{m_{k}} \right)=0,$$ that is, 
\[
\Phi \left( a_{1},\ldots, a_{k} \right)=
\Phi \left(
f_{1} \left( a^{(1)}_{1},\ldots, a^{(1)}_{m_{1}} \right),
\ldots,
f_{k} \left( a^{(k)}_{1},\ldots, a^{(k)}_{m_{k}} \right) \right)=0.
\]
Therefore, $\Phi$ preserves zeros of $f_{0}$. Since $f_1(C_1) \times \cdots \times f_k(C_k)$ is $f_{0}$-zpd,  there exists a linear functional $\tau \colon A \to F$ such that, for all $a_{i} \in f_i \left(C_i \right)$,
\[
\Phi \left( a_{1},\ldots, a_{k} \right) =\tau \left( f_{0} \left( a_{1}, \ldots, a_{k}  \right) \right).
\]
The proof is complete since, for all $a^{(i)}_{j} \in A$,
\[
\varphi \left( 
a^{(1)}_{1},\ldots, a^{(k)}_{m_{k}}
\right)
=
\tau \left(
f \left( 
a^{(1)}_{1},\ldots, a^{(k)}_{m_{k}}
 \right) 
\right).
\qedhere\]
\end{proof}

 The following corollary is immediate.

\begin{corollary} \label{fzpdcorollary}
 Let  $f_1, \ldots, f_{k} $
be multilinear polynomials in distinct variables and let $f_0$ be a multilinear polynomial in $k$ variables. Let $A$ be an $F$-algebra satisfying the following two conditions:
\begin{enumerate}
    \item[(a)] $A$ is $f_i$-zpd for 
$i=0,1,\dots, k$.
\item[(b)]  $f_i(A)=A$  for  $i=1,\dots, k$.
\end{enumerate}
 Then the algebra $A$ is $f$-zpd where 
$
f= f_{0} \left( f_{1},\dots,f_{k} \right).
$
\end{corollary}

The applicability of  Theorem \ref{invariantunderendomorphisms} and Corollary \ref{fzpdcorollary} of course depends on the validity of the L'vov-Kaplansky conjecture
and its variants.

We continue with a lemma needed for another corollary to Theorem \ref{invariantunderendomorphisms}.
Recall that ${\rm sl}_d(F)$ stands for the Lie algebra of trace zero matrices in $M_d(F)$.

\begin{lemma} \label{commutatorzpd}
  Let $f_0 = [x_1, x_2]$ and let $B_1,B_2\in \{{\rm sl}_d(F), M_d(F)\}$. Then the set
  $B_1\times B_2$ is $f_0$-zpd,
  %for ${\rm sl}_d(F) \times M_d(F)$, $M_d(F) \times {\rm sl}_d(F)$ and  ${\rm sl}_d(F) \times {\rm sl}_d(F)$, 
  provided that
{\rm char($F$)} is $0$ or does not divide $d$.
\end{lemma}
\begin{proof}
We know that $M_d(F)$ is zLpd (see Theorem \ref{zLpdzJpd}), which means that $M_d(F)\times M_d(F)$ is $f_0$-zpd.

Let $\varphi \colon {\rm sl}_d(F) \times M_d(F) \to F$ be a bilinear functional preserving zeros of $f$. Our assumption on char$(F)$ implies that $M_d(F) = {\rm sl}_d(F) \oplus F\cdot 1$. Therefore, we  extend $\varphi$ to $M_d(F)^2$ by setting $\varphi(1, a) = 0$ for all $a \in M_d(F)$. Let now $a, b \in M_d(F)$ be such that $[a,b]=0$. Writing $a = [a_1, a_2] + \lambda 1$ with $a_1 , a_2 \in M_d(F)$ and $\lambda \in F$, we thus have $[[a_1,a_2],b]=0$. Since $[a_1,a_2]\in {\rm sl}_d(F)$,
it follows that
$$
 \varphi(a,b) = \varphi([a_1,a_2],b) = 0.
$$
As $M_d(F)$ is zLpd, there exists a linear functional $\tau \colon M_d(F) \to F$ such that $\varphi(c,d)=\tau([c,d])$ for all $c,d \in M_d(F)$, and so in particular for all $c \in {\rm sl}_d(F)$ and $ d \in M_d(F)$. This proves that the set ${\rm sl}_d(F)\times M_d(F)$
is $f_0$-zpd.

The $M_d(F) \times {\rm sl}_d(F)$ case
can be handled similarly, and so can be the 
${\rm sl}_d(F) \times {\rm sl}_d(F)$ case. Indeed,  one extends a bilinear functional $\varphi$ defined on $ {\rm sl}_d(F) \times {\rm sl}_d(F)$ to $M_d(F)^2$ by setting
$
\varphi(1,a)=\varphi(a,1)=0$
for all $ a \in M_d(F)$.
\end{proof}

\begin{corollary} \label{multliemonomialszpd}
If $f$ is a multilinear Lie monomial, then the algebra $M_{d}(F)$ is $f$-zpd, provided that
{\rm char($F$)} is $0$ or does not divide $d$.
\end{corollary}
\begin{proof}
First let us show that $f(M_{d}(F))$ is a vector space. In fact, we claim that  $f(M_{d}(F))={\rm sl}_d(F)$, unless the degree $m$ of $f$ is $1$, in which case  $f(M_{d}(F))$ is obviously equal to $M_{d}(F)$.
We may therefore assume that  $m>1$ and that our claim is true for Lie monomials of degree less than $m$. Write $f=[f_{1},f_{2}]$, where $f_{1}$ and $f_2$ are multilinear Lie monomials in distinct variables of degree at most $m-1$. By 
\cite{AlbertMuckenhoupt}, every matrix in 
sl$_d(F)$ is  a commutator of two matrices 
from $M_d(F)$. However,
since $M_d(F) = {\rm sl}_d(F) \oplus F\cdot 1$
by the characteristic assumption, it is actually a commutator of two matrices from
sl$_d(F)$. Since, by our assumption, 
$f_1(M_d(F))$ and $f_2(M_d(F))$ contain ${\rm sl}_d(F)$, it follows that $f(M_{d}(F))={\rm sl}_d(F)$.

Let us now prove that $M_d(F)$ is $f$-zpd. There is nothing to prove if $m=1$, so
we may assume that $m> 1$ and that by writing 
$f=[f_1,f_2]$ as above we have that 
$M_d(F)$ is $f_i$-zpd, $i=1,2$.
Since $f_{i}(M_{d}(F))\in \{{\rm sl}_d(F), M_d(F)\}$, taking into account Lemma \ref{commutatorzpd} we can apply Theorem \ref{invariantunderendomorphisms} to conclude that $M_{d}(F)$ is $f$-zpd.
\end{proof}

It is clear that the method of proof can be used for some other polynomials. For example, using the fact that the algebra $M_d(F)$ is zJpd (Theorem \ref{zLpdzJpd})
and that the polynomial
$f_0=x_1x_2+x_2x_1$ obviously satisfies $f_0(M_d(F)) = M_d(F)$ provided that char$(F)\ne 2$, we see that Corollary \ref{fzpdcorollary} yields the following
result (by a  Jordan monomial we mean an element of the free special Jordan algebra).

\begin{corollary} \label{l}
Let {\rm char($F$)}$\ne 2$.
If $f$ is a multilinear Jordan monomial, then the algebra $M_{d}(F)$ is $f$-zpd.
\end{corollary}

\subsection{Polynomials given by cyclic permutations} 

In this subsection we deal with a multilinear polynomial $f$
whose monomials correspond to a cyclic permutation and satisfies the condition that the sum of its coefficients is nonzero. The only assumption that we will require on our algebra is that it is generated by idempotents. Our result therefore in particular holds for the matrix algebra $M_d(S)$ where $S$ is any unital algebra, see \cite[Corollary 2.4]{Bresar2021}. 
Its special cases where 
$f=x_1x_2$ or 
$f=x_1x_2 + x_2x_1$ are well known, see the first statement of Theorem \ref{zLpdzJpd}. 
In our proof we will use ideas from the proof concerning the second case, i.e., the proof
that an algebra generated by idempotents is zJpd (see  \cite{ALH} or \cite[Theorem 3.15]{Bresar2021}).

Let us now state the theorem.

\begin{theorem} \label{Zanpolynomial}
Let {\rm char$(F)$} $\ne 2$, let $\alpha_1,\dots,\alpha_m\in F$ be such that
$\sum_{i=1}^m \alpha_i\ne 0$,
and let
$$
f(x_1, \ldots, x_m) = \alpha_1 x_1 \cdots x_m + \alpha_2 x_2 \cdots x_m x_1 +  \cdots + \alpha_m x_m x_1 \cdots x_{m-1}.$$
If an $F$-algebra $A$ is generated by idempotents, then $A$ is $f$-zpd.
\end{theorem}

\begin{proof} In light of Remark \ref{zpdproportionalpolynomials}, we may assume without loss of generality that $\sum_{i=1}^m \alpha_i=1$.
Let $\varphi$ be an $m$-linear functional preserving zeros of $f$.
By Lemma \ref{equivalentconditiontau}, it suffices to prove that
 $\varphi$ satisfies 
\begin{equation}
    \label{zz}
\varphi(a_1,\ldots, a_m) =  \varphi(f(a_1,\ldots,a_m),1,\ldots,1)
\end{equation}
for all $a_1,\dots,a_m\in A$. 

Set $\alpha_{m+1}=\alpha_1$. We claim that
$\alpha_i + \alpha_{i+1} \ne 0$ for some 
$i$. Indeed, if $m$ is even then this is immediate from 
$$1 = (\alpha_1+\alpha_2) +\dots + (\alpha_{m-1}+\alpha_m), $$
and if $m$ is odd this follows from the observation that
\[
    \alpha_1 = -\alpha_2 = \cdots = \alpha_m = -\alpha_1
\]
implies $2\alpha_1=0$ and hence $\alpha_1=0$, so we can again use the assumption that the sum of all $\alpha_i$ is $1$. After relabeling, if necessary, we may assume that $i=m$, i.e., $$\alpha_1+\alpha_m\ne 0.$$

Let $S$ denote the set of all
$s\in A$ such that
$$\varphi(a_1,\ldots, a_{m-1},s) = \varphi(f(a_1,\ldots, a_{m-1},s), 1, \ldots,1)$$
for all $a_1,\dots,a_{m-1}\in A$. To prove (\ref{zz}), we have to show that $S=A$.
We will establish this by induction on $m$. 

In the base case where $m=2$ we have $$f(x_{1},x_{2})=\alpha_{1}x_{1}x_{2}+\alpha_{2}x_{2}x_{1}$$ with $\alpha_{1}+\alpha_2=1$. The proof that we will give is just a minor modification of the proof that $A$ is zJpd (i.e., of the case 
where $\alpha_1=\alpha_2 $). However, we provide details for  the sake of completeness.

Since $S$ is a vector subspace of $A$, it is enough to show that $S$ contains every product of the form $e_{1}e_{2}\cdots e_{n}$ where $e_{i}$ are idempotents in $A$. We proceed by induction on $n$. Let $e\in A$ be an idempotent, and let us first prove that $e\in S$. Considering an arbitrary $a\in A$ and writing $h=1-e$, we have
\[
a=eae+hae+eah+hah
\] 
and hence
\[
\varphi(a,e)=\varphi(eae,e)+\varphi(hae,e)+\varphi(eah,e)+\varphi(hah,e).
\]
Recalling that $\alpha_1+\alpha_2=1$, one can easily check that 
\[
f(hae,e-\alpha_1)=f(eah,e-\alpha_2)=f(eae,h)=f(hah,e)=0,
\]
which gives us the following relations
\begin{equation*}\label{relatiosdegree2}
\begin{aligned}
& \varphi(hae,e) = \alpha_1\varphi(hae,1), \\
& \varphi(eah,e) = \alpha_2\varphi(eah,1), \\
& \varphi(eae,e) = \varphi(eae,1), \\
& \varphi(hah,e) = 0.
\end{aligned}
\end{equation*}
Consequently,
\begin{equation*}
\begin{aligned}
\varphi(a,e) &= 
\varphi(eae,e)+\varphi(eah,e)+\varphi(hae,e)+\varphi(hah,e)  \\ & =
\varphi(eae,1)+\alpha_1\varphi(hae,1)+\alpha_2\varphi(eah,1)
\\ & =
\varphi(eae+\alpha_1 hae+\alpha_2 eah,1)\\ &
=\varphi(f(a,e),1).
\end{aligned}
\end{equation*}
We have thus shown that $e\in S$.  

Next, assuming that $S$ contains products of $n$ idempotents, let us prove that $S$ also contains $e_{1}\cdots e_{n}e_{n+1}$, where each $e_{i}\in A$ is an idempotent. Write
\[
h_{1}=1-e_{1}, h_{n+1}=1-e_{n+1}, t=e_{2}\cdots e_{n},
\]
($t=1$ if $n=1$), so we want  to prove that $e_1te_{n+1}\in S$. For any $a\in A$, we have
\[
a=e_{n+1}ae_1+h_{n+1}ae_{1}+e_{n+1}ah_{1}+h_{n+1}ah_{1}.
\]
Therefore,
\begin{equation*}
\begin{aligned}
\varphi(a,e_{1}te_{n+1})&= \varphi(e_{n+1}ae_{1},e_{1}te_{n+1})+\varphi(h_{n+1}ae_{1},e_{1}te_{n+1})\\
 & \ \ \ +\varphi(e_{n+1}ah_{1},e_{1}te_{n+1})+\varphi(h_{n+1}ah_1,e_{1}te_{n+1})\\
&=\varphi(e_{n+1}ae_{1},te_{n+1}-h_{1}t+h_{1}th_{n+1})\\
& \ \ \ +\varphi(h_{n+1}ae_{1},te_{n+1}-h_{1}te_{n+1})\\
& \ \ \ +\varphi(e_{n+1}ah_{1},e_{1}t-e_{1}tf_{n+1})\\
&\ \ \ +\varphi(h_{n+1}ah_{1},e_{1}ae_{n+1}). 
\end{aligned}
\end{equation*}

Since
\begin{equation*}
\begin{aligned}
f(e_{n+1}ae_{1},h_{1}th_{n+1})=0,\\
f(h_{n+1}ae_{1},h_{1}ae_{n+1})=0,\\
f(e_{n+1}ah_{1},e_{1}ah_{n+1})=0,\\
f(h_{n+1}ah_{1},e_{1}ae_{n+1})=0
\end{aligned}
\end{equation*}
and hence 
\begin{equation*}
\begin{aligned}
\varphi(e_{n+1}ae_{1},h_{1}th_{n+1})=0,\\
\varphi(h_{n+1}ae_{1},h_{1}ae_{n+1})=0,\\
\varphi(e_{n+1}ah_{1},e_{1}ah_{n+1})=0,\\
\varphi(h_{n+1}ah_{1},e_{1}ae_{n+1})=0,
\end{aligned}
\end{equation*}
we can conclude that
\begin{equation*}
\begin{aligned}
\varphi(a,e_1 te_{n+1})&=\varphi(e_{n+1}ae_{1},te_{n+1}-h_{1}t)\\
&\ \ \ +\varphi(h_{n+1}ae_{1},te_{n+1})+\varphi(e_{n+1}ah_{1},e_{1}t).
\end{aligned}
\end{equation*} Since $te_{n+1}, h_1t, te_{n+1}, e_1t$ lie in $S$ by the induction 
 hypothesis, it follows   that
\begin{eqnarray*}
\begin{aligned}
&\varphi(e_{n+1}ae_{1},te_{n+1})=\varphi(\alpha_{1}e_{n+1}ae_{1}te_{n+1}+\alpha_{2} te_{n+1}ae_{1},1),\\
&\varphi(e_{n+1}ae_{1},h_{1}t)=\varphi(\alpha_{2} h_{1}te_{n+1}ae_{1},1),\\
&\varphi(h_{n+1}ae_{1},te_{n+1})=\varphi(\alpha_{1}h_{n+1}ae_{1}te_{n+1},1),\\
&\varphi(e_{n+1}ah_{1},e_{1}t)=\varphi(\alpha_{2} e_{1}te_{n+1}ah_{1},1).\\
\end{aligned}
\end{eqnarray*}
One easily checks that this implies that
\[
\varphi(a,e_{1}te_{n+1})=\varphi(f(a,e_{1}te_{n+1}),1).
\]
Hence $e_{1}te_{n+1}\in S$, which concludes the proof for the base case where $m=2$.

Assume now $m > 2$ and that the result holds for $m-1$, and so in particular for the polynomial $f(x_1,\dots,x_{m-1},1)$ (which is also of the required form). As  the $(m-1)$-linear functional $\varphi(a_1,\dots,a_{m-1},1)$ preserves its zeros, we have 
\begin{equation}
    \label{zz2}
\varphi(a_1,\ldots, a_{m-1},1) =  \varphi(f(a_1,\ldots,a_{m-1},1),1,\ldots,1)
\end{equation}
for all $a_1,\dots,a_{m-1}\in A$. 

As in the $m=2$ case, it is enough to show that $S$ contains every element of the form $e_1e_2\cdots e_n$ where $e_i$ are idempotents in $A$.
The proof that we will give is conceptually similar to the one just given, but the necessary changes are non-obvious.
We proceed by induction on $n$.

To handle the base case, take an idempotent $e = e_1 \in A$. We must prove that $e \in S$. Denote $1 - e$ by $h$ and write  $a_1 = e a_1 + h a_1$ and $ a_{m-1} = a_{m-1}e + a_{m-1}h $. Thus,
\begin{equation} \label{longform}
\begin{split}
\varphi(a_1,\ldots, a_{m-1},e)  =& \varphi(e a_1,\ldots, a_{m-1}e,e) + \varphi(e a_1,\ldots, a_{m-1}h,e)\\
&+\varphi(h a_1,\ldots, a_{m-1}e,e) + \varphi(h a_1,\ldots, a_{m-1}h,e).
\end{split}
\end{equation}
%
%(if $m = 2$ then write $A_1 = eA_1e + eA_1h + hA_1e + hA_1h$ and it continues similarly).
It is easy to see that 
\begin{equation} \label{someconditions}
\begin{aligned}
& f(ea_1,\ldots, a_{m-1}e,e - 1) = 0, \\
& f(ea_1,\ldots, a_{m-1}h,(\alpha_1 + \alpha_m )e - \alpha_m 1)  = 0, \\
& f(ha_1,\ldots, a_{m-1}e,(\alpha_1 + \alpha_m )e - \alpha_1 1) = 0, \\
& f(ha_1,\ldots, a_{m-1}h,e) = 0.
\end{aligned}
\end{equation}
Of course, $\varphi$ then satisfies the same identities, which can be written as 
\begin{equation*}
\begin{aligned}
& \varphi(ea_1,\ldots, a_{m-1}e,e) = \varphi(ea_1,\ldots, a_{m-1}e,1), \\
& (\alpha_1 + \alpha_m)\varphi(ea_1,\ldots, a_{m-1}h,e) = \alpha_m  \varphi(ea_1,\ldots, a_{m-1}h,1), \\
& (\alpha_1 + \alpha_m)\varphi(ha_1,\ldots, a_{m-1}e,e) = \alpha_1  \varphi(ha_1,\ldots, a_{m-1}e,1) , \\
& \varphi(ha_1,\ldots, a_{m-1}h,e) = 0.
\end{aligned}
\end{equation*}
Consequently,
 (\ref{longform}) becomes
\begin{equation*}
\begin{aligned}
\varphi(a_1,\ldots, a_{m-1},e) &= 
\varphi(e a_1,\ldots, a_{m-1}e,1)  \\
& \ \ \ + \alpha_m (\alpha_1 + \alpha_m)^{-1}  \varphi(e a_1,\ldots, a_{m-1}h,1) \\ & \ \ \ + \alpha_1 ( \alpha_1 + \alpha_m)^{-1}
\varphi( h a_1,\ldots, a_{m-1}e,1).\\
\end{aligned}
\end{equation*}
Applying (\ref{zz2})  it follows that
\begin{equation*}
\begin{aligned}
\varphi(a_1,\ldots, a_{m-1},e) &= 
\varphi(f(e a_1,\ldots, a_{m-1}e,1),1,\ldots,1)  \\ & \ \ \ + 
\alpha_m(\alpha_1+\alpha_m)^{-1}\varphi(f(ea_1,\ldots,a_{m-1}h,1),1,\ldots,1)
\\ & \ \ \ + 
\alpha_1(\alpha_1 + \alpha_m)^{-1}\varphi(f( h a_1,\ldots, a_{m-1}e,1),1,\ldots,1).
\end{aligned}
\end{equation*}
Using (\ref{someconditions}) we obtain
% \\ & \ \ \ + 
%\varphi(f(h a_1,\ldots, a_{m-1}h,e),1,\ldots,1)\Bigr)\\ &=
\begin{equation*}
\begin{aligned}
\varphi(a_1,\ldots, a_{m-1},e) &=
\varphi(f(e a_1,\ldots, a_{m-1}e,e),1,\ldots,1) \\ & \ \ \ + 
\varphi(f(e a_1,\ldots, a_{m-1}h,e),1,\ldots,1)\\ & \ \ \ + 
\varphi(f(h a_1,\ldots, a_{m-1}e,e),1,\ldots,1).
\end{aligned}
\end{equation*}
Since $f(h a_1,\ldots, a_{m-1}h,e)=0$ (see (\ref{someconditions})), it follows that
\begin{equation*}
\begin{aligned}
\varphi(a_1,\ldots, a_{m-1},e)
=
 \varphi(f(a_1,\ldots, a_{m-1},e),1,\ldots,1).
\end{aligned}
\end{equation*}
This means that $e\in S$, as desired.

 We may now assume that  any product of $n$ idempotents is contained in $S$. Take idempotents
 $e_1,\dots,e_{n+1}$ and let us
 prove that $S$ contains $e_1 \cdots e_n e_{n+1}$. Write
\[
h_1 = 1 - e_1, \ \ h_{n+1} = 1 - e_{n+1}, \ \ t = e_2\cdots e_n 
\]
 ($t=1$ if $n=1$).
We have to show that $e_1te_{n+1}\in S$. Take  $a_1, a_2,\ldots, a_{m-1}\in A$ and write
\[
a_1 = e_{n+1} a_1 + h_{n+1}a_1 \text{ \ and \ } a_{m-1} = a_{m-1}e_1 + a_{m-1}h_1.
\]
We have
\begin{equation*}
\begin{aligned}
\varphi(a_1,\ldots, a_{m-1},e_1te_{n+1}) &= 
\varphi(e_{n+1} a_1,\ldots, a_{m-1}e_1,e_1te_{n+1})  \\ & \ \ \ + 
\varphi(e_{n+1} a_1,\ldots, a_{m-1}h_1,e_1te_{n+1}) \\ & \ \ \ + 
\varphi(h_{n+1} a_1,\ldots, a_{m-1}e_1,e_1te_{n+1})  \\ & \ \ \ + 
\varphi(h_{n+1} a_1,\ldots, a_{m-1}h_1,e_1te_{n+1}) \\ &= 
\varphi(e_{n+1} a_1,\ldots, a_{m-1}e_1,te_{n+1} - h_1t + h_1th_{n+1})  \\ & \ \ \ + 
\varphi(e_{n+1} a_1,\ldots, a_{m-1}h_1,e_1t-e_1th_{n+1}) \\ & \ \ \ + 
\varphi(h_{n+1} a_1,\ldots, a_{m-1}e_1,te_{n+1}-h_1te_{n+1})   \\ & \ \ \ + 
\varphi(h_{n+1} a_1,\ldots, a_{m-1}h_1,e_1te_{n+1}).
\end{aligned}
\end{equation*}
One easily checks that
\begin{equation}\label{moj}
\begin{aligned}
f(e_{n+1} a_1,\ldots, a_{m-1}e_1,h_1th_{n+1}) & = 0, \\
f(e_{n+1} a_1,\ldots, a_{m-1}h_1,e_1th_{n+1}) & = 0, \\ 
f(h_{n+1} a_1,\ldots, a_{m-1}e_1,h_1te_{n+1}) & = 0, \\ 
f(h_{n+1} a_1,\ldots, a_{m-1}h_1,e_1te_{n+1}) & = 0.
\end{aligned}
\end{equation}
As $\varphi$ then satisfies the same identities,
it follows that 
%Using the induction assumption we get
\begin{equation*}
\begin{aligned}
\varphi(a_1,\ldots, a_{m-1},e_1te_{n+1}) & = 
\varphi(e_{n+1} a_1,\ldots, a_{m-1}e_1,te_{n+1} - h_1t) \\ & \ \ \ \ \ \ + 
\varphi(e_{n+1} a_1,\ldots, a_{m-1}h_1,e_1t)\\ & \ \ \ \ \ \ + 
\varphi(h_{n+1} a_1,\ldots, a_{m-1}e_1,te_{n+1}) .
\end{aligned}
\end{equation*}
Since $te_{n+1}, h_1t, e_1t, te_{n+1}\in S$ by the induction assumption, it follows that
\begin{equation*}
\begin{aligned}
\varphi(a_1,\ldots, a_{m-1},e_1te_{n+1})  & = 
\varphi(f(e_{n+1} a_1,\ldots, a_{m-1}e_1,te_{n+1} - h_1t),1,\ldots,1) \\ & \ \ \ \ \ \ + \varphi(f(e_{n+1}a_1,\ldots, a_{m-1}h_1,e_1t),1,\ldots,1)  \\ & \ \ \ \ \ \ + 
\varphi(f(h_{n+1}a_1,\ldots, a_{m-1}e_1,te_{n+1}),1,\ldots,1).
\end{aligned}
\end{equation*}
Applying (\ref{moj}) we 
finally obtain
\begin{equation*}
\begin{aligned}
& \varphi(a_1,\ldots, a_{m-1},e_1te_{n+1}) \\ &  \ \ \  = 
\varphi(f(e_{n+1} a_1,\ldots, a_{m-1}e_1,te_{n+1} - h_1t + h_1th_{n+1}),1,\ldots,1)  \\ & \ \ \ \ \ \ +
\varphi(f(e_{n+1} a_1,\ldots, a_{m-1}h_1,e_1t-e_1th_{n+1}),1,\ldots,1) \\ & \ \ \ \ \ \ + \varphi(f(h_{n+1} a_1,\ldots, a_{m-1}e_1,te_{n+1}-h_1te_{n+1}),1,\ldots,1) \\ &  \ \ \ \ \ \ + \varphi(f(h_{n+1} a_1,\ldots, a_{m-1}h_1,e_1te_{n+1}),1,\ldots,1) \\ & \ \ \ = 
 \varphi(f( a_1,\ldots, a_{m-1},e_1te_{n+1}),1,\ldots,1).
\end{aligned}
\end{equation*}
This means that $e_1te_{n+1}\in S$ and the proof is complete. 
\end{proof}

\subsection{The generalized commutator}

This last subsection is devoted to the {\em generalized commutator}
$$
f(x_1, x_2, x_3)=x_1 x_2 x_3 - x_3 x_2 x_1.
$$
This is one of the  polynomials that deserve  special attention (see, e.g., \cite{KL}), so the question of whether  the  algebra $M_d(F)$ is $f$-zpd occurs naturally. We will show that the answer is affirmative.

Throughout this subsection, we assume that
$\varphi \colon M_d(F)^3\to F$ is a $3$-linear functional such that for all $a,b,c\in M_d(F)$,
\begin{eqnarray}\label{hypothesis}
abc-cba=0 \implies \varphi(a,b,c)=0.
\end{eqnarray}
Our goal is to prove that $\varphi$ satisfies the condition presented in Lemma  \ref{equivalentconditiontau2}. Thus,
assume that $N\ge 1$ and that the matrices
\[
a^{(t)}=\sum_{i,j=1}^{d}a_{ij}^{t}e_{ij}, \ \ \ \ \ \ \
b^{(t)}=\sum_{i,j=1}^{d}b_{ij}^{t}e_{ij}, \ \ \ \ \ \ \
c^{(t)}=\sum_{i,j=1}^{d}c_{ij}^{t}e_{ij},
\]
$t=1,\dots,N$, where $e_{ij}$ are standard matrix units,
satisfy
\begin{equation}
    \label{hh}
 \sum_{t=1}^{N} a^{(t)} b^{(t)} c^{(t)} - c^{(t)} b^{(t)} a^{(t)}=0.\end{equation}
 We have to show that 
\begin{equation}
    \label{hh2}
\sum_{t=1}^{N} \varphi \left(a^{(t)}, b^{(t)}, c^{(t)} \right)=0.\end{equation}

We proceed by a series of lemmas.

\begin{lemma}\label{newhypothesis}
We have
\begin{equation*} 
\sum_{t=1}^{N} \left( \sum_{l=1}^{d}\sum_{\substack{k=1\\ 
k\neq j}}^{d}a_{ik}^{t}b_{kl}^{t}c_{lj}^{t}-c_{ik}^{t}b_{kl}^{t}a_{lj}^{t} +\sum_{\substack{l=1\\
l\neq i}}^{d}a_{ij}^{t}b_{jl}^{t}c_{lj}^{t}-c_{ij}^{t}b_{jl}^{t}a_{lj}^{t} \right) = 0.
\end{equation*}
\end{lemma}

\begin{proof}
Note that for each pair $(i,j)\in\{1,\dots,d\}^{2}$, 
\[
\left( a^{(t)} b^{(t)} c^{(t)} - c^{(t)} b^{(t)} a^{(t)} \right)_{ij}=\sum_{k,l=1}^{d} a_{ik}^{t}b_{kl}^{t}c_{lj}^{t}-c_{ik}^{t}b_{kl}^{t}a_{lj}^{t},
\]
and hence, by (\ref{hh}),
\[
\sum_{t=1}^{N}\sum_{k,l=1}^{d} a_{ik}^{t}b_{kl}^{t}c_{lj}^{t}-c_{ik}^{t}b_{kl}^{t}a_{lj}^{t}=0.
\]
Clearly $a_{ij}^{t}b_{ji}^{t}c_{ij}^{t}-c_{ij}^{t}b_{ji}^{t}a_{ij}^{t}=0$ for all $t$. Hence this sum reduces to the one from the statement of the lemma.
\end{proof}
It is obvious that
$f(a,b,a) =0$ and so 
$$\varphi(a,b,a)=0,$$ yielding $$\varphi(a,b,c)=-\varphi(c,b,a)$$ for all $a,b,c\in M_d(F)$. In what follows, we will use these two identities  without comment.
\begin{lemma}\label{nova}
We have
\begin{align*}
&\sum_{t=1}^{N} \varphi \left(a^{(t)}, b^{(t)}, c^{(t)} \right)\\ =& 
\sum_{t=1}^{N} 
\sum_{i,j=1}^{d} 
\Bigg( \sum_{l=1}^{d} \sum_{\substack{k=1\\  k\neq j}}^{d} 
a_{ik}^{t} b_{kl}^{t} c_{lj}^{t} 
\varphi(e_{ik},e_{kl},e_{lj})  +
\sum_{\substack{l=1 \\ l\neq i}}^{d} 
a_{ij}^{t} b_{jl}^{t} c_{lj}^t 
\varphi(e_{ij},e_{jl},e_{lj}) \\
 - &\sum_{l=1}^{d} \sum_{\substack{k=1\\
k\neq j}}^{d}c_{ik}^{t}b_{kl}^{t}a_{lj}^{t} \varphi(e_{ik},e_{kl},e_{lj}) 
 -\sum_{\substack{l=1\\ l\neq i}}^{d}c_{ij}^{t} b_{jl}^{t} a_{lj}^{t} \varphi(e_{ij},e_{jl},e_{lj}) \Bigg).
\end{align*}

\end{lemma}
\begin{proof}
 Clearly,
\begin{equation} \label{goal}
\begin{aligned}
\sum_{t=1}^{N}\varphi \left( a^{(t)},b^{(t)},c^{(t)} \right) &= 
\sum_{t=1}^{N}\varphi \left( \sum_{i,j=1}^{d}a_{ij}^{t}e_{ij},\sum_{k,l=1}^{d}b_{kl}^{t}e_{kl},\sum_{p,q=1}^{n}c_{pq}^{t}e_{pq} \right) \\
& = \sum_{t=1}^{N}\sum_{i,j,k,l,p,q=1}^{d}a_{ij}^{t}b_{kl}^{t}c_{pq}^{t}\varphi(e_{ij},e_{kl},e_{pq}).
\end{aligned}
\end{equation}
It is easy to check that if  $i,j,k,l,p,q$ satisfy one of the following conditions 
\begin{align*}\nonumber
j\neq k \ &\mbox{and} \ q\neq k,\\\nonumber
j\neq k \ &\mbox{and} \ i\neq l,\\\nonumber
l\neq p \ &\mbox{and} \ q\neq k,\\\nonumber
l\neq p \ &\mbox{and} \ i\neq l,
\end{align*}
then $f(e_{ij},e_{kl},e_{pq})=0$ and so $\varphi(e_{ij},e_{kl},e_{pq})=0$.
Hence we may assume that the following relations hold:
\begin{align*}
j=k \ &\mbox{or} \ q=k,\\\nonumber
j=k \ &\mbox{or} \ i=l,\\\nonumber
l=p \ &\mbox{or} \ q=k,\\\nonumber
l=p \ &\mbox{or} \ i=l.
\end{align*}
We can rewrite  (\ref{goal}) as
\begin{eqnarray*}
\begin{aligned}
\sum_{t=1}^{N}\varphi \left( a^{(t)}, b^{(t)}, c^{(t)} \right)  &= 
\sum_{t=1}^{N}\sum_{i,j,l,q=1}^{d}a_{ij}^{t}b_{jl}^{t}c_{lq}^{t}\varphi(e_{ij},e_{jl},e_{lq}) \\ & \ \ \ + \sum_{t=1}^{N}\sum_{i,j,p=1}^{d}\sum_{\substack{k=1 \\
k\neq j}}^{d}a_{ij}^{t}b_{ki}^{t}c_{pk}^{t}\varphi(e_{ij},e_{ki},e_{pk}) \\
& \ \ \ + \sum_{t=1}^{N}\sum_{i,j=1}^{d}\sum_{\substack{p=1\\
p\neq i}}^{d}a_{ij}^{t}b_{ji}^{t}c_{pj}^{t}\varphi(e_{ij},e_{ji},e_{pj}).
\end{aligned}
\end{eqnarray*}
 Hence,
\begin{equation*}
\begin{split}
\sum_{t=1}^{N}\varphi \left( a^{(t)}, b^{(t)}, c^{(t)} \right)
& = \sum_{t=1}^{N} \sum_{i,j=1}^{d} \Bigg( \sum_{k,l=1}^{d}a_{ik}^{t}b_{kl}^{t}c_{lj}^{t}\varphi(e_{ik},e_{kl},e_{lj}) \\ 
& \ \ \ -\sum_{l=1}^{d}\sum_{\substack{k=1\\
k\neq j}}^{d}c_{ik}^{t}b_{kl}^{t}a_{lj}^{t}\varphi(e_{ik},e_{kl},e_{lj}) \\
& \ \ \  -\sum_{\substack{l=1\\ l\neq i}}^{d}c_{ij}^{t}b_{jl}^{t}a_{lj}^{t}\varphi(e_{ij},e_{jl},e_{lj}) \Bigg).
\end{split}
\end{equation*}
Finally, using $\varphi(e_{ij},e_{ji},e_{ij})=0$ we obtain the statement of the lemma.
\end{proof}

\begin{lemma}\label{fi1}
If  $u\neq l, i$, $l\neq i$, and $j\neq k$, then 
$$\varphi(e_{ik},e_{kl},e_{lj})=\varphi(e_{ik},e_{ku},e_{uj}).$$
\end{lemma}
\begin{proof}
Note that $$f(e_{ik}+e_{uj},e_{kl}+e_{ku},e_{lj}+e_{ik})=0.$$ Therefore,
\begin{align*}
0 
&= \varphi(e_{ik}+e_{uj},e_{kl}+e_{ku},e_{lj}+e_{ik}) \\
&= \varphi(e_{ik},e_{kl},e_{lj})+\varphi(e_{ik},e_{kl},e_{ik}) +\varphi(e_{ik},e_{ku},e_{lj})  \\ & \ \ \ +\varphi(e_{ik},e_{ku},e_{ik}) 
+ \varphi(e_{uj},e_{kl},e_{lj}) +\varphi(e_{uj},e_{kl},e_{ik}) \\ & \ \ \ +\varphi(e_{uj},e_{ku},e_{lj})+\varphi(e_{uj},e_{ku},e_{ik}) \\
&= \varphi(e_{ik},e_{kl},e_{lj})+\varphi(e_{uj},e_{ku},e_{ik}). \qedhere
\end{align*}
\end{proof}

\begin{lemma}\label{fi2}
If  $l\neq i$ and $k\neq j$, then $$\varphi(e_{ij},e_{jl},e_{lj})=\varphi(e_{ik},e_{kl},e_{lj}).$$
\end{lemma}
\begin{proof}
Note that
$$
f(e_{ij}+e_{lj},e_{jl}+e_{kl},e_{lj}+e_{ik})=0.
$$
Therefore,
\begin{align*}
0 &= \varphi(e_{ij}+e_{lj},e_{jl}+e_{kl},e_{lj}+e_{ik}) \\
&= \varphi(e_{ij},e_{jl},e_{lj})+\varphi(e_{ij},e_{jl},e_{ik}) +\varphi(e_{ij},e_{kl},e_{lj}) \\ & \ \ \ + \varphi(e_{ij},e_{kl},e_{ik})  
 + \varphi(e_{lj},e_{jl},e_{lj})+\varphi(e_{lj},e_{jl},e_{ik}) \\ & \ \ \ +\varphi(e_{lj},e_{kl},e_{lj})+\varphi(e_{lj},e_{kl},e_{ik}) \\
&= \varphi(e_{ij},e_{jl},e_{lj})+\varphi(e_{lj},e_{kl},e_{ik}).\qedhere
\end{align*}
\end{proof}

\begin{lemma}\label{fi3}
If $k\neq i$ and $k\neq j$, then $$\varphi(e_{ik},e_{ki},e_{ij})=\varphi(e_{ik},e_{kk},e_{kj}).$$
\end{lemma}
\begin{proof}
Note that
$$
f(e_{ik}+e_{kj},e_{ki}+e_{kk},e_{ij}+e_{ik})=0.
$$
Therefore,
\begin{align*}
0 &= \varphi(e_{ik}+e_{kj},e_{ki}+e_{kk},e_{ij}+e_{ik}) \\
&= \varphi(e_{kk},e_{ki},e_{ij})+\varphi(e_{ik},e_{ki},e_{ik})+\varphi(e_{ik},e_{kk},e_{ij}) \\ & \ \ \ +\varphi(e_{ik},e_{kk},e_{ik})
+ \varphi(e_{kj},e_{ki},e_{ij})+\varphi(e_{kj},e_{ki},e_{ik}) \\ & \ \ \ +\varphi(e_{kj},e_{kk},e_{ij})+\varphi(e_{kj},e_{kk},e_{ik}) \\
&= \varphi(e_{ik},e_{ki},e_{ij})+\varphi(e_{kj},e_{kk},e_{ik}).
\qedhere\end{align*}
\end{proof}

\begin{lemma}\label{fi4}
    If $k\neq i$, then
$$\varphi(e_{ii},e_{ik},e_{ki})=\varphi(e_{ik},e_{kk},e_{ki}).$$
\end{lemma}

\begin{proof}
    Note that 
    $$f(e_{ii}+e_{ki},e_{ik}+e_{kk},e_{ki}+e_{ik})=0$$ and therefore
    \begin{align*}
0 &= \varphi(e_{ii}+e_{ki},e_{ik}+e_{kk},e_{ki}+e_{ik}) \\
&= \varphi(e_{ii},e_{ik},e_{ki})+\varphi(e_{ii},e_{ik},e_{ik})+\varphi(e_{ii},e_{kk},e_{ki}) \\ & \ \ \ +\varphi(e_{ii},e_{kk},e_{ik})
+ \varphi(e_{ki},e_{ik},e_{ki})+\varphi(e_{ki},e_{ik},e_{ik}) \\ & \ \ \ +\varphi(e_{ki},e_{kk},e_{ki})+\varphi(e_{ki},e_{kk},e_{ik}) \\
&= \varphi(e_{ii},e_{ik},e_{ki})+\varphi(e_{ki},e_{kk},e_{ik}).
\qedhere\end{align*}
\end{proof}

\begin{lemma}\label{fi5} We have 
$$\varphi(e_{ii},e_{ij},e_{jj})=\varphi(e_{ii},e_{ii},e_{ij})=\varphi(e_{ij},e_{jj},e_{jj}).$$
\end{lemma}
\begin{proof}
Note that
$$
f(e_{ii}+e_{jj},e_{ii}+e_{ij},e_{ij}+e_{ii})=0.$$
Therefore,
\begin{align*}
0 &= \varphi(e_{ii}+e_{jj},e_{ii}+e_{ij},e_{ij}+e_{ii}) \\
& = \varphi(e_{ii},e_{ii},e_{ij})+\varphi(e_{ii},e_{ii},e_{ii})+\varphi(e_{ii},e_{ij},e_{ij}) \\ & \ \ \ +\varphi(e_{ii},e_{ij},e_{ii})  + \varphi(e_{jj},e_{ii},e_{ij})+\varphi(e_{jj},e_{ii},e_{ii}) \\ & \ \ \ +\varphi(e_{jj},e_{ij},e_{ij})+\varphi(e_{jj},e_{ij},e_{ii}) \\
&= \varphi(e_{ii},e_{ii},e_{ij})+\varphi(e_{jj},e_{ij},e_{ii}).
\end{align*}
The second equality can be obtained analogously.\qedhere
\end{proof}

The next lemma gathers together all the needed information from the previous lemmas.

\begin{lemma}\label{corollaryresume}
    For all $i,j\in\{1,\dots,n\}$,
    the set $$\Phi_{ij}=\{\varphi(e_{ik},e_{kl},e_{lj})\,|\,k,l=1\dots,n,\, k\ne j\mbox{ or } l\ne i\}$$ is a singleton. 
\end{lemma}

\begin{proof}
    Assume first that $i\neq j$. We claim that 
    $$\Phi_{ij}=\{\varphi(e_{ij},e_{jj},e_{jj})\}.$$
    Consider $\varphi(e_{ik},e_{kl},e_{lj})$ where $k\neq j$ and $l\neq i$. Since $j\neq i$, by Lemma \ref{fi1} we have $\varphi(e_{ik},e_{kl},e_{lj})=\varphi(e_{ik},e_{kj},e_{jj})$. We now apply Lemma \ref{fi2} to get $\varphi(e_{ik},e_{kj},e_{jj})=\varphi(e_{ij},e_{jj},e_{jj})$.

We now consider the case where $l\neq i$ and $k=j$. Taking $u\neq j$, then Lemma \ref{fi2} shows that $\varphi(e_{ij},e_{jl},e_{lj})=\varphi(e_{iu},e_{ul},e_{lj})$. However, since $u\neq j$, $l\neq i$, we see from the the previous case that $\varphi(e_{iu},e_{ul},e_{lj})=\varphi(e_{ij},e_{jj},e_{jj})$.

We now consider the case where $j\neq k$ and $i=l$. If $k=i$, then by Lemma \ref{fi5} we have $\varphi(e_{ii},e_{ii},e_{ij})=\varphi(e_{ij},e_{jj},e_{jj})$. If $k\neq i$, then  $\varphi(e_{ik},e_{ki},e_{ij})=\varphi(e_{ik},e_{kk},e_{kj})$. Now it is enough to apply the first case to get $\varphi(e_{ik},e_{kk},e_{ki})=\varphi(e_{ij},e_{jj},e_{jj})$, as desired.

We may now consider the case where $i=j$. Fix $u\neq i$. We claim that    
\[
\Phi_{ii}=\{\varphi(e_{ii},e_{iu},e_{ui})\}.
\]

Assume first that $k\neq i$ and $l\neq i$. Then, by Lemma \ref{fi1}, $\varphi(e_{ik},e_{kl},e_{li})=\varphi(e_{ik},e_{ku},e_{ui})$ and, by Lemma \ref{fi2}, $\varphi(e_{ik},e_{ku},e_{ui})=\varphi(e_{ii},e_{iu},e_{ui})$.

If $k\neq i$ and $l=i$, then Lemma \ref{fi3} gives $\varphi(e_{ik},e_{ki},e_{ii})=\varphi(e_{ik},e_{kk},e_{ki})$. By the previous case we have $\varphi(e_{ik},e_{kk},e_{ki})=\varphi(e_{ii},e_{iu},e_{ui})$.

Finally, let  $k=i$ and $l\neq i$. By Lemma \ref{fi4},  $\varphi(e_{ii},e_{il},e_{li})=\varphi(e_{il},e_{ll},e_{li})$, and once again we have $\varphi(e_{il},e_{ll},e_{li})=\varphi(e_{ii},e_{iu},e_{ui})$.
\end{proof}

We now arrive at  the goal of  this subsection. 

\begin{theorem} \label{existenceoftau} Let $f = x_1 x_2 x_3 - x_3 x_2 x_1$.
 Then the algebra  $M_d(F)$ is $f$-zpd. 
\end{theorem}
\begin{proof}
As already mentioned, in light of
 Lemma \ref{equivalentconditiontau2}
 it is enough to prove (\ref{hh2}). Consider the right-hand side of the identity given in Lemma \ref{nova}.
Denoting $\Phi_{ij}=\{\varphi_{ij}\}$ by Lemma \ref{corollaryresume}, we have

\begin{align*}
\sum_{t=1}^{N}\varphi \left( a^{(t)}, b^{(t)}, c^{(t)} \right) 
&= \sum_{t=1}^{N}\sum_{i,j=1}^{d} \Bigg( \sum_{l=1}^{d}\sum_{\substack{k=1\\
k\neq j}}^{d}a_{ik}^{t}b_{kl}^{t}c_{lj}^{t}-c_{ik}^{t}b_{kl}^{t}a_{lj}^{t} \\ & \ \ \ +\sum_{\substack{l=1\\
l\neq i}}^{d}a_{ij}^{t}b_{jl}^{t}c_{lj}^{t}-c_{ij}^{t}b_{jl}^{t}a_{lj}^{t} \Bigg)\varphi_{ij} \\
&= \sum_{i,j=1}^{d}\sum_{t=1}^{N} \Bigg( \sum_{l=1}^{d}\sum_{\substack{k=1\\
k\neq j}}^{d}a_{ik}^{t}b_{kl}^{t}c_{lj}^{t}-c_{ik}^{t}b_{kl}^{t}a_{lj}^{t} \\ & \ \ \ +\sum_{\substack{l=1\\
l\neq i}}^{d}a_{ij}^{t}b_{jl}^{t}c_{lj}^{t}-c_{ij}^{t}b_{jl}^{t}a_{lj}^{t} \Bigg)\varphi_{ij}.
\end{align*}
Invoking   Lemma \ref{newhypothesis} we now obtain the desired conclusion that
\begin{equation*}
\sum_{t=1}^{N}\varphi \left( a^{(t)}, b^{(t)}, c^{(t)} \right) =0.\qedhere\end{equation*}
\end{proof}

\section{A multilinear Nullstellensatz}\label{s4}

In this last section we consider the situation where $f$ and $g$ are multilinear polynomials of the same degree $m$ such that every zero of $f$ in $A^m$, where $A$ is an algebra, is also a zero of $g$;
 that is, for all $a_1,\dots,a_m\in A$,
$$f(a_1,\dots,a_m)=0\implies g(a_1,\dots,a_m)=0.$$
As we mentioned in the introduction,  this is a special case of the condition from 
Amitsur's Nullstellensatz \cite{Am}. On the other hand,
from Lemma \ref{l35} it is evident  that it is  also a special case of the condition from the definition of an $f$-zpd algebra. 
It is therefore natural to ask whether the problem of describing the relation between $f$ and $g$ can be solved in any $f$-zpd algebra. In the next proposition, we give a positive answer under the assumption that $f(1,\dots,1)\ne 0$.  It is not clear to us at present whether this assumption can be removed.

\begin{proposition}\label{pp}
Let $A$ be an  $F$-algebra and let $f,g\in F\langle x_1,x_2,\dots\rangle$ be multilinear polynomials of  degree $m$ such that every zero of $f$ in $A^m$ is a zero of $g$. If  $A$ is $f$-zpd and  $f(1,\dots,1)\neq0$, then there exist a $\lambda\in F$ and  a 
polynomial identity $h$ of $A$ such that $g=\lambda f+h$.
\end{proposition}

\begin{proof}
It is clear from our assumptions that  
$$(a_{1},\dots,a_{m})\mapsto g(a_{1},\dots,a_{m})$$ 
is an $m$-linear map that preserves zeros of $f$.
 Lemma \ref{l35} therefore shows that there exists a linear map $T \colon A\to A$ satisfying $$g(a_{1},\dots,a_{m})=T(f(a_{1},\dots,a_{m})).$$ Thus, for every $a\in A$ we have
\[
g(1,\dots,1)a=g(a,1,\dots,1)=T(f(a,1,\dots,1))=f(1,\dots,1)T(a).
\]
Setting $\lambda=g(1,\dots,1)f(1,\dots,1)^{-1}$ we thus have
$T(a)=\lambda a$ for all  $a\in A$, and hence
 $$g(a_{1},\dots,a_{m})=\lambda f(a_{1},\dots,a_{m})$$ for all $a_1,\dots,a_m \in A$. This means that $h=g-\lambda f$ is a polynomial identity of $A$.
\end{proof}

\begin{corollary} \label{Zanpolynomial2}
Let {\rm char$(F)$} $\ne 2$, let $\alpha_1,\dots,\alpha_m\in F$ be such that
$\sum_{i=1}^m \alpha_i\ne 0$,
and let
$$
f(x_1, \ldots, x_m) = \alpha_1 x_1 \cdots x_m + \alpha_2 x_2 \cdots x_m x_1 +  \cdots + \alpha_m x_m x_1 \cdots x_{m-1}.$$
If an $F$-algebra $A$ is generated by idempotents and $g$ is a multilinear polynomial of degree $m$ such that every zero of $f$ in $A^m$ is a zero of $g$, 
 then there exist a $\lambda\in F$ and  a  polynomial identity $h$ of $A$ such that $g=\lambda f+h$.
\end{corollary}

\begin{proof}
This is immediate from Theorem \ref{Zanpolynomial} and Proposition \ref{pp}.
\end{proof}

From now on we consider the case where $A=M_d(F)$. We recall from Proposition \ref{propertyofg} that the conclusion of Proposition \ref{pp} then does not always hold. Our goal is to show that it does hold if $m < 2d-3$. In fact, since, as is well known, $M_d(F)$  has no polynomial identities of degree less than $2d$, we will actually prove that $f$ and $g$ are linearly dependent. To this end, we start by introducing  the necessary notation.

In what follows, let
$\alpha_\sigma,\beta_\sigma\in F$ be such that
$$
f =  \sum_{\sigma\in S_m} \alpha_\sigma x_{\sigma(1)} \cdots x_{\sigma(m)},
$$
$$
g =  \sum_{\sigma\in S_m} \beta_\sigma x_{\sigma(1)} \cdots x_{\sigma(m)}.
$$
We set
$$\text{Supp}(f)=\{\sigma\in S_m \ | \ \alpha_\sigma \ne 0 \}$$
and similarly we define Supp$(g)$. Further, for each $\sigma\in S_m$ we write
$$(x_1,\ldots, x_m)_\sigma = x_{\sigma(1)} \cdots x_{\sigma(m)}.$$
Thus,  $$f=
\sum_{\sigma\in S_m} \alpha_\sigma (x_1,\ldots, x_m)_\sigma.$$ We will consider evaluations 
$(a_1,\ldots,a_m)_\sigma$
with $a_i\in M_d(F)$.

The next two definitions are standard in group theory.

\begin{definition}
 We define a metric 
$d$ on $S_m$ by letting
$d(\sigma_1,\sigma_2)$ to be the smallest nonnegative integer $k$ for which there exists a sequence of transpositions $\tau_1, \tau_2, \ldots, \tau_k \in S_m$  such that 
$
\tau_k \cdots  \tau_1 \sigma_1 = \sigma_2.
$
\end{definition}

The next definition concerns any subset $T$ of  $S_m$, but we  will be actually interested in the case where  $T= {\rm Supp}(f)$.

\begin{definition}
Let $T$ be a subset of $S_m$. We define an equivalence relation on $T$ as follows: $\sigma_1 \sim \sigma_2$ if and only if there exists a (possibly empty) sequence of transpositions $\tau_1, \tau_2 , \ldots, \tau_k$ such that 
\begin{itemize}
\item[(a)] $\tau_i  \cdots \tau_1 \sigma_1\in T$, $i = 1,\ldots,k-1$,
\vspace{0.1 cm}
\item[(b)] $\tau_k  \cdots  \tau_1  \sigma_1 = \sigma_2$.
\end{itemize}    
\end{definition}

The following theorem is  the main result of this section.

\begin{theorem} \label{Null}Let $f,g\in F\langle x_1,x_2,\dots\rangle$ be multilinear polynomials of  degree $m$ such that every zero of $f$ in $M_d(F)^m$ is a zero of $g$.
If 
 $m < 2d-3$, then there exists a $\lambda \in F$ such that $ f = \lambda g$.
\end{theorem}

\begin{proof} 
Set $m_0 = \frac{m}{2} +1$ if $m$ is even and $m_0 = \frac{m+1}{2}$ if $m$ is odd.
Note that $m_0+1\le d$ since 
$m+3 < 2d$.
Define a sequence $\bm{e}=(e_1,\dots,e_m)$ by setting
\[
e_1 = e_{11}, \ \ \ e_2 = e_{12}, \ \ \ e_3 = e_{22}, \ \ \ \ldots, \ \ \ 
e_m = \begin{cases}
  e_{m_0 -1, m_0}  & m \text{ even } \\
  e_{m_0, m_0} & m \text{ odd }
\end{cases}.
\]
For any $\sigma, \pi \in S_m$, we have 
$$
(e_{\sigma(1)}, \ldots, e_{\sigma(m)})_\pi = e_{\pi(\sigma(1))} \cdots e_{\pi(\sigma(m))} = \begin{cases}
  e_{1,m_0}  & \mbox{if } \sigma = \pi^{-1} \\
  0 & \text{otherwise}
\end{cases}.
$$
Therefore, for every $\sigma\in S_{m}$ we have
\begin{equation*}
    \begin{split}
        \alpha_\sigma = 0 & \implies f \left(e_{\sigma^{-1}(1)},\ldots,e_{\sigma^{-1}(m)} \right) = \alpha_\sigma e_{1, m_0} = 0  \\
        & \implies g \left(e_{\sigma^{-1}(1)},\ldots,e_{\sigma^{-1}(m)} \right) = \beta_\sigma e_{1, m_0} = 0 \\ & \implies \beta_\sigma = 0.
    \end{split}
\end{equation*}
We have thereby proved that \begin{equation}\label{es}
\supp(g)\subseteq \supp(f).\end{equation}

Take $\sigma\in {\rm Supp}(f)$ and write $\lambda = \alpha_\sigma^{-1}\beta_\sigma$, so that
$\beta_\sigma = \lambda\alpha_\sigma$.
 We claim that
 \begin{equation}\label{ecla}
 \beta_{\tau \sigma} = \lambda\alpha_{\tau\sigma}
 \end{equation}
 for every transposition
 $\tau$. Indeed, without loss of generality we may assume that $\sigma=(1)$ and we write $\tau=(p\,q)$ with $p<q$. We 
 consider four cases.
\\
  
{\bf Case 1:} $p$ and $q$ are both even. In this case $e_p$ and $e_q$ are square-zero matrices. Hence, considering the matrices
 
 \begin{equation*}
    \begin{split}
       a_{i} &= e_{i}, \  i \in\{1, \ldots, m\}\backslash \{p,q\}, \\
    a_{p} &= e_{p} + e_{q}, \\
    a_{q} &= \alpha_{(1)}e_{p} -\alpha_{\tau}e_{q},
    \end{split}
\end{equation*}
we  have 
$$
f(a_1,\ldots,a_m) = (\alpha_{\tau}\alpha_{(1)}-\alpha_{(1)}\alpha_{\tau})e_{1,m_0}=0.$$ 
This implies that
$$
0=g(a_1,\ldots,a_m) = (\beta_{\tau}\alpha_{(1)}-\beta_{(1)}\alpha_{\tau}) e_{1,m_0}= (\beta_{\tau}-\lambda\alpha_{\tau})\alpha_{(1)}e_{1,m_0},
$$
which yields (\ref{ecla}).

{\bf Case 2:} $p$ and $q$ are both odd. In this case both $e_p$ and $e_q$ are idempotents. We reduce this case to the previous one by considering a shift on the sequence $\bm{e}$, that is,
\[
\tilde{e}_1 = e_{12}, \quad \tilde{e}_2 = e_{22}, \quad \tilde{e}_3 = e_{23}, \quad  \ldots, \quad  
\tilde{e}_m = \begin{cases}
  e_{m_0, m_0},  & m \text{ even } \\
  e_{m_0, m_0+1}, & m \text{ odd }
\end{cases}.
\]
Now it is enough to perform the  evaluation at
 \begin{equation*}
    \begin{split}
       a_{i} &= \tilde{e}_{i}, \  i \in\{1, \ldots, m\}\backslash \{p,q\}, \\
    a_{p} &= \tilde{e}_{p} + \tilde{e}_{q}, \\
    a_{q} &= \alpha_{(1)}\tilde{e}_{p} -\alpha_{\tau}\tilde{e}_{q}
    \end{split}
\end{equation*}
and proceed as at the end of Case 1.

{\bf Case 3:} $p$ is odd and $q$ is even. In this case  $e_p$ is an idempotent but $e_q$ is not. The idea now is to consider a shift on the sequence from $e_p$ on. This shift will turn the element in the $q$-th position into an idempotent. So an additional change will be needed in this element as well. Precisely we take
\begin{eqnarray*}
\tilde{e}_1 = e_{11}, \quad \dots, \quad \tilde{e}_{p-1} = e_{\frac{p-1}{2},\frac{p-1}{2}+1}, \quad \tilde{e}_p = e_{\frac{p+1}{2},\frac{p+1}{2}+1}, \\
\tilde{e}_{p+1}=e_{\frac{p+1}{2}+1,\frac{p+1}{2}+1},\quad \dots, \tilde{e}_{q-1}=e_{\frac{q}{2},\frac{q}{2}+1}, \quad \tilde{e}_q=e_{\frac{q}{2}+1,\frac{q}{2}+2},\\
\tilde{e}_{q+1}=e_{\frac{q}{2}+2,\frac{q}{2}+2},\quad \dots, \quad \tilde{e}_{m} = \begin{cases}
  e_{m_0, m_0+1},  & m \text{ even } \\
  e_{m_0+1, m_0+1}, & m \text{ odd }.
\end{cases}
\end{eqnarray*}
We  consider the evaluation at
\begin{equation*}
    \begin{split}
       a_{i} &= \tilde{e}_{i}, \  i \in\{1, \ldots, m\}\backslash \{p,q\}, \\
    a_{p} &= \tilde{e}_{p} + \tilde{e}_{q}, \\
    a_{q} &= \alpha_{(1)}\tilde{e}_{p} -\alpha_{\tau}\tilde{e}_{q}
    \end{split}
\end{equation*}
and once again we proceed as in the end of Case 1.

{\bf Case 4:} $p$ is even and $q$ is odd. Here we have that $e_p$ as a square-zero matrix and $e_q$ is idempotent. We proceed similarly as in the previous case. The difference, however, is that no shift is needed at the beginning, just the change in the elements at the $q$-th position. 

This completes the proof of our claim.\\

Let $[\sigma]$ denote the equivalence class of $\sigma$ in $\supp(f)/\sim$. Write $\lambda_{[\sigma]}$ for $\lambda$. Observe that 
(\ref{ecla}) implies that
\begin{equation}
    \label{inv}
\beta_{\sigma'} = \lambda_{[\sigma]} \alpha_{\sigma'}\end{equation}
for all $\sigma'\in [\sigma]$. 

In view of (\ref{es}) and  (\ref{inv}),
we are left to prove that $\lambda_A=\lambda_B$ for all equivalence classes $A,B$ in ${\rm \supp}(f)/\sim$.
Assume this is not true and consider a pair of permutations $\sigma_1$ and $\sigma_2$ such that
\begin{equation}\label{min1}
d(\sigma_1, \sigma_2) = \min_{\substack{\pi_1\in A, \pi_2\in B \\ A,B\in \supp(f)/\sim\\ \lambda_A\ne \lambda_B}}d(\pi_1, \pi_2) =: \ell.
\end{equation}

Let $f_\psi$ denote the reindexing of $f$ through the permutation $\psi$, i.e.,
$$f_{\psi}=\sum_{\sigma\in S_{m}} \alpha_{\sigma}x_{\psi\sigma(1)}\cdots x_{\psi\sigma(m)}.$$
It is not difficult to prove that
\begin{equation}\nonumber
\ell=\min_{\substack{\pi_1 \in A,\pi_2 \in B \\ A,B \in \supp(f_{\psi})/\sim \\ \lambda_A \neq \lambda_{B}}}d(\pi_{1},\pi_{2}).
\end{equation}
This means that the minimum $\ell$ in (\ref{min1}) is invariant under reindexing of the variables in both $f$ and $g$.

Reindexing, if necessary, we may therefore assume that
$\sigma_2 \sigma_1^{-1}$ is the product of disjoint cycles 
\[
\sigma_2 \sigma_1^{-1}= (s_1 \ s_1-1\,\cdots\,1) \cdots (s_h \ s_h-1 \,\cdots \, s_{h-1}+1).
\]
Since disjoint cycles commute, we may also assume that the first $p$ cycles are of even length and the remaining $h-p$ are of odd length, where $0\le  p \le h$. 

Setting $s_0 = 0$ and taking into account for instance \cite{George}, we have  
\[
\ell=s_h-h=\sum_{i=1}^{h} s_{i}-s_{i-1}-1.
\]

Finally,  we also assume that $\sigma_1= (1)$. 
Only minor adjustments in the proof are needed if 
 $\sigma_1$ is an arbitrary permutation, which, however, make the reading more difficult.
Thus, from now one we will be dealing with the permutations $(1)$ and $$
\sigma_2 = (s_1 \ s_1-1\,\cdots\,1) \cdots (s_h \ s_h-1 \,\cdots \, s_{h-1}+1).
$$ We have  $\lambda_{[(1)]} \ne \lambda_{[\sigma_2]}$.

In order to obtain a contradiction, our goal will be to construct a sequence $E \in M_d(F)^m$ such that 
$$
E_{(1)} = e_{1,m_{0}}, \ \ \ 
E_{\sigma_2} = -\dfrac{\alpha_{(1)}}{\alpha_{\sigma_2}}e_{1,m_{0}}, \ \ \
E_{\sigma} = 0 \ \mbox{for all} \ \sigma \in \supp(f)\setminus{ \{(1),\sigma_2\}}.
$$
This will imply
\[
f(E) = \alpha_{(1)}e_{1,m_{0}}-\alpha_{\sigma_2}\frac{\alpha_{(1)}}{\alpha_{\sigma_2}}e_{1,m_{0}} = 0,
\]
hence
\[
g(E) = \lambda_{[(1)]}\alpha_{(1)}e_{1,m_{0}}-\lambda_{[\sigma_2]}\alpha_{\sigma_2}\frac{\alpha_{(1)}}{\alpha_{\sigma_2}}e_{1,m_{0}} = (\lambda_{[(1)]} - \lambda_{[\sigma_2]})\alpha_{(1)}e_{1,m_{0}} = 0,
\]
ans so $\lambda_{[(1)]} = \lambda_{[\sigma_2]}$,  contrary to the assumption. 

%%%%%%%%%%%%%%%%%%%%%%%%%%%%%%%%%%%%
%Write $U = S_m \backslash \supp(f)$. 
%%%%%%%%%%%%%%%%%%%%%%%%%%%%%%%%%%%%

%Consider the following transpositions:
%\begin{align*}
%&\tau_i = (i \ i+1), i = 1, \ldots, s_1 -1,
%\\
%&\tau_j = (j+1 \ j+2), j = s_1, \ldots, s_2 -1,
%\\
%&\tau_\ell = (s_h-1 \ s_h).
%\end{align*}
%
%Moreover, for $\tau_1 = (1,2)$, $\tau_2 = (2,3), \ldots, \tau_{s_1-1} = (s_1 - 1, s_1)$, $\tau_{s_1} = (s_1+1,s_1+2)$, $\ldots$, $\tau_{\ell} = (s_h-1, s_h)$,
%We get that $$
%\tau_\ell  \cdots  \tau_1 \sigma_1 = \sigma_2.
%$$

%Write $U = S_m \backslash \supp(f)$.
%%

\smallskip

As the first step, we shall construct the sequence $E$ in three particular cases. We will see at the end that the general case will follow from these three ones.

\smallskip
{\bf Case 1:} $p=h$ (all cycles are of even length). 

\smallskip

By assumption,  $s_i = 2q_i$,  $i = 1, \ldots, h$.  We introduce the sequence $E$  in blocks as follows:
$$
E = (E_1, \ldots, E_h, E_{h+1}),
$$
\begin{itemize}
\item[•] $E_1 = e_{11},e_{12},e_{22},\ldots,e_{q_1-1,q_1},e_{q_1,q_1}+e_{q_1,q_1+1}, e_{q_1,q_1+1}+e_{11},$
\vspace{0.1 cm}
\item[•] $E_2 = e_{q_1+1,q_1+1},e_{q_1+1,q_1+2},\ldots, e_{q_2-1,q_2}, e_{q_2,q_2}+e_{q_2,q_2+1}$, \\  $ \hspace*{1 cm} e_{q_2,q_2+1}+e_{q_1+1,q_1+1}$, \\
$\hspace*{0.2 cm} \vdots$
\vspace{0.1 cm}
\item[•] $E_h = e_{q_{h-1}+1,q_{h-1}+1}, \ldots, e_{q_h-1,q_h},e_{q_h,q_h}+e_{q_h,q_h+1}$, \\ $\hspace*{1 cm} e_{q_h,q_h+1}-\dfrac{\alpha_{(1)}}{\alpha_{\sigma_2}}e_{q_{h-1}+1,q_{h-1}+1},$
\item[•] $E_{h+1} = e_{q_{h}+1,q_{h}+1}, e_{q_{h}+1,q_{h}+2}, \ldots, e_{m',m_0}$.
\end{itemize}
Here $m' = m_0$ if $m$ is odd and $m' = m_0-1$ otherwise (recall that $m_0$ is defined at the beginning of the proof).

%Let
%\begin{equation*}
%    \begin{split}
%        E=&  (e_{11},e_{12},e_{22},\ldots,e_{q_1-1,q_1},e_{q_1,q_1}+e_{q_1,q_1+1}, e_{q_1,q_1+1}+e_{11}, \\
%        & e_{q_1+1,q_1+1},e_{q_1+1,q_1+2},\ldots e_{q_2-1,q_2}, e_{q_2,q_2}+e_{q_2,q_2+1},  e_{q_2,q_2+1}+e_{q_1+1,q_1+1},\\& \ldots,\\& e_{q_{h-1}+1,q_{h-1}+1},  e_{q_{h-1}+1,q_{h-1}+2},\ldots, e_{q_h-1,h_2},e_{q_h,q_h}+e_{q_h,q_h+1}, \\ &\,\,\,\,\,\,\,
%    e_{q_h,q_h+1}-\dfrac{\alpha_{\sigma_1}}{\alpha_{\sigma_2}}e_{q_{h-1}+1,q_{h-1}+1},\\ &e_{q_{h}+1,q_{h}+1}, e_{q_{h}+1,q_{h}+2}, \ldots).
%    \end{split}
%\end{equation*}
%Note that $E$ consists of $h+1$ ``blocks'', with the first one ending with $e_{q_1,q_1 +1}+e_{11}$, the second one ending with $e_{q_2,q_2 +1}+e_{q_1 +1,q_1 +1}$, and so on until the $h$-th block ending with $e_{q_h,q_h +1}+e_{q_{h-1}+1,q_{h-1}+1}$. The $(h+1)$-th block is defined as the remaining matrix units at the end of the sequence $E$. 

One can easily see that the nonzero products of matrices in $E$ are only obtained by joining the nonzero evaluations of each block, in the increasing order of the blocks.

For $i = 1,2,\ldots, h$, define the following sets of permutations:
$$
R_i = \{ (1), (s_i \, \cdots \, s_{i-1} + 3 \, \ s_{i-1} + 2), (s_i \, \cdots \, s_{i-1} + 2 \ \, s_{i-1} + 1)\}.
$$

Then we get
\[
E_\sigma = \begin{cases}
  e_{1,m_{0}} & \mbox{if} \ \sigma=(1) \\
  -\frac{\alpha_{(1)}}{\alpha_{\sigma_2}}e_{1,m_{0}} & \mbox{if} \ \sigma=\sigma_2 \\
  \mu_\sigma e_{1,m_{0}} & \mbox{if} \ \sigma = \pi_1  \cdots \pi_h, \  \pi_i\in R_i, \ \sigma \not \in \{ (1), \sigma_2 \}  \\
  0 & \text{otherwise,}
\end{cases}
\]
where $ \mu_\sigma \in F$.

In order to complete the proof of this case we are left to show that actually the permutations $\eta$ of the third item of $E_\sigma$ giving a nonzero evaluation are not elements of $\supp(f)$. To this end let us prove the following facts. 
\begin{itemize}
\item[•] $d(\eta, (1)) < \ell$.

\smallskip
\noindent This follows from a direct comparison between $\eta$ and $\sigma_2$. Indeed, let $\eta = \pi_1 \cdots \pi_h$, $\pi_i \in R_i$. Since $\eta \ne \sigma_2$, we have that at least one of the $\pi_i$'s is not equal to $ (s_i \, \cdots \,  s_{i-1} + 2 \, \ s_{i-1} + 1)$. Hence we have a fewer number of transpositions in the decomposition of $\eta$ than in that  of $\sigma_2$. As a consequence we obtain $d(\eta, (1)) < \ell$, as desired.

\vspace{0.2 cm}

\item[•] $d(\eta, \sigma_2) < \ell$.

\smallskip
\noindent As before let $\eta = \pi_1 \cdots \pi_h$, $\pi_i \in R_i$. Since $\eta \ne (1)$, we have that at least one of the $\pi_i$'s is not equal to $ (1)$. Now consider $\sigma_2  \eta^{-1}$. If $\eta$ involves a cycle of the form $(s_{i} \,\dots \, s_{i-1}+2 \, \ s_{i-1}+1)$, then the elements from the set $\{s_{i},\dots,s_{i-1}+2,s_{i-1}+1\}$ are fixed in $\sigma_2  \eta^{-1}$. The outcome of this is that, with respect to the $i$-th block, we have fewer transpositions in $\sigma_2 \eta^{-1}$ than in $\sigma_2$ and we are done in this case. The other possibility is that a cycle of the form $(s_{i} \,\cdots \, s_{i-1}+3 \, \ s_{i-1}+2)$ occurs in $\eta$. In this case, in the $i$-th block of $\sigma_{2} \eta^{-1}$ only the transposition $(s_{i-1}+1 \ s_{i})$ appears. Since $1< 3\leq s_{i}-s_{i-1}-1$, we reach the desired conclusion $d(\eta, \sigma_2) < \ell$.
\end{itemize}

According to the above two facts we can complete the proof of this case. If $\lambda_{[\eta]} \ne \lambda_{[(1)]} $, we immediately get a contradiction to the minimality of $\ell$ since we have proved that $d(\eta, (1)) < \ell$. So assume that $\lambda_{[\eta]} = \lambda_{[(1)]} $. Since by hypothesis $\lambda_{[\sigma_2]} \ne \lambda_{[(1)]} $, we obtain that $\lambda_{[\eta]} \ne \lambda_{[\sigma_2]} $. Again we get a contradiction to the minimality of $\ell$ since we have proved that $d(\eta, \sigma_2) < \ell$.

%
%
%We claim that, if $\pi_1  \cdots \pi_h\notin \{(1),\sigma_2\}$, for $\pi_i\in R_i$, then $\pi_1  \cdots \pi_h\in U$. Indeed, let $\eta$ be a permutation with nonzero evaluation on $E$ such that $\eta\notin\{(1),\sigma_2\}$. Recall that $\ell=\sum_{i=1}^{h} s_{i}-s_{i-1}-1$. Of course $d(\eta,(1))< \ell $ (it follows from direct comparison between $\eta$ and $\sigma_2$.) Now let us prove that $d(\eta,\sigma_2)< \ell$.  Taking into account $\sigma_2  \eta^{-1}$, we see that if in $\eta$ occurs a cycle of the form $(s_{i},\dots, s_{i-1}+2,s_{i-1}+1)$, then the elements from the set $\{s_{i},\dots,s_{i-1}+2,s_{i-1}+1\}$ are fixed in $\sigma_2  \eta^{-1}$ (that is, with respect to the $i$-th block, we have less transpositions in $\sigma_2 \eta^{-1}$ than in $\sigma_2$). If occurs a cycle of the form $(s_{i},\dots,s_{i-1}+3,s_{i-1}+2)$, then in $\sigma_{2} \eta^{-1}$ will appear only one transposition $(s_{i-1}+1,s_{i})$ in the $i$-th block, and here we get that $1< 2\leq s_{i}-s_{i-1}-1$. Finally, if the identity permutation appears in the $i$-th block, then $(s_{i},\dots,s_{i-1}+2,s_{i-1}+1)$ must appear in $\sigma_2  \eta^{-1}$. Since at least one non trivial permutation from some $R_{i}$ must occur in $\eta$, this gives us that $d(\sigma_2,\eta)< \ell$. We have proved that $d(\eta,(1))< \ell$ and $d(\eta,\sigma_2)< \ell$. Now it is enough to use the minimality of $\ell$ to get $\eta \in U$.

\smallskip
{\bf Case 2:} $h=1$ and $s_1 = 2q + 1$ is odd. 
\smallskip

Consider the sequence $E = (E_1, E_2)$ given in two blocks as follows:
\begin{itemize}
\item[•] $E_1 = e_{11},e_{12},e_{22},\ldots,e_{q,q+1}, e_{q+1,q+1}-\dfrac{\alpha_{(1)}}{\alpha_{\sigma_2}}e_{11},$
\vspace{0.1 cm}
\item[•] $E_2 = e_{q+1,q+2},e_{q+2,q+3},\ldots, e_{m',m_0}$.
\end{itemize}

%Let $s_1 = 2q+1$. Define
%\[
%E = (e_{11},e_{12},e_{22},\ldots,e_{q,q+1}, e_{q+1,q+1}-\frac{\alpha_{\sigma_1}}{\alpha_{\sigma_2}}e_{11},e_{q+1,q+2},e_{q+2,q+3},\ldots).
%\]
%Note that $E$ is given by two blocks, where the first one ends with the difference of matrix units. 

It is not difficult to see that
\[
E_\sigma = \begin{cases}
e_{1,m_{0}}, & \mbox{if} \ \sigma=(1) \\
-\frac{\alpha_{(1)}}{\alpha_{\sigma_2}}e_{1,m_{0}}, & \mbox{if} \ \sigma\in\{\sigma_2,(s_{1} \, \cdots \,  3 \ 2)\} \\
 0, & \text{otherwise}
\end{cases}.
\]
The proof of this case is complete since it is sufficient to observe that $(s_1 \, \cdots \,  3 \ 2) \not \in \supp(f)$.

\smallskip
{\bf Case 3:}  $h=2$, $p=0$ ($2$ odd cycles). 
\smallskip

In this case we have that $s_1 = 2q_1+1$ and $s_2 = 2 q_2$. 

Consider the sequence $E = (E_1, E_2)$ given in two blocks as follows:
\begin{itemize}
\item[•] $E_1 = e_{11},e_{12},e_{22},\ldots,e_{q_1,q_1+1}, e_{q_1+1,q_1+1}+e_{11}, e_{q_1+1,q_1+2}, \\ 
\hspace*{1 cm} e_{q_1+2,q_1+2},\ldots, e_{q_2-1,q_2},e_{q_2,q_2}+e_{q_2,q_2+1}, e_{q_2,q_2+1}-\dfrac{\alpha_{(1)}}{\alpha_{\sigma_2}}e_{q_1+1,q_1+1},$
\item[•] $E_2 = e_{q_2+1,q_2+1}, e_{q_2+1,q_2+2}, \ldots, e_{m',m_0}$.
\end{itemize}

%
%
%Let $2q_1+1=s_1$, $2q_2=s_2$. Define
%\begin{equation*}
%    \begin{split}
%        E =& (e_{11},e_{12},e_{22},\ldots,e_{q_1,q_1+1}, e_{q_1+1,q_1+1}+e_{11}, 
%         e_{q_1+1,q_1+2}, \\ & e_{q_1+2,q_1+2},\ldots e_{q_2-1,q_2},e_{q_2,q_2}+e_{q_2,q_2+1}, e_{q_2,q_2+1}-\frac{\alpha_{\sigma_1}}{\alpha_{\sigma_2}}e_{q_1+1,q_1+1}, \\ & e_{q_2+1,q_2+1}, e_{q_2+1,q_2+2}, \ldots ).
%    \end{split}
%\end{equation*}
%Here we consider $E$ in two blocks (the first one ends with the difference of matrix units).

Define the following two sets of permutations:
\begin{equation*}
\begin{split}
R_1 =& \ \{ (1), (s_1 \, \cdots \, 3 \ 2), (s_1 \, \cdots \, 2 \ 1)\},\\
R_2 =& \ \{(1), (s_{2} \, \cdots \, s_{1}+2 \ s_{1}+1)\}.
\end{split}
\end{equation*}

One can directly see  that 
\[
E_\sigma = \begin{cases}
  e_{1,m_{0}}, & \mbox{if} \ \sigma=(1) \\
  -\frac{\alpha_{(1)}}{\alpha_{\sigma_{2}}}e_{1,m_{0}}, & \mbox{if} \ \sigma=\sigma_2 \\
  \mu_\sigma e_{1,m_0},  & \mbox{if} \ \sigma = \pi_1 \pi_2, \ \pi_i\in R_i, \ \sigma \not \in \{ (1), \sigma_2 \}  \text{ or }  \sigma = (s_2 \, \cdots \, s_{1})
  \\
  0, & \text{otherwise},
\end{cases}
\]
where $\mu_\sigma \in F$.

Let $\eta$ be a permutation of the third item of $E_\sigma$ giving a nonzero evaluation. In order to complete the proof of this case we need to show that  $\eta \not \in \supp(f)$. Assume first that $\eta = \pi_1 \pi_2$, where $\pi_i$'s are permutations from $R_{i}$. In this case we are in the  situation of Case 1 and, proceeding in a similar manner, we arrive at the desired conclusion. Now suppose that $\eta = (s_{2} \, \cdots \, s_{1}+1 \ s_{1})$. In this case we obtain that 
$$
\sigma_2  \eta^{-1} = (s_{1} \ s_{2} \ s_{1}-1 \ s_{1}-2 \,\cdots\, s_{0}+2 \ s_{0}+1),
$$
which decomposes into $s_{1}-s_{0}$ transpositions. Since $s_{2}-s_{1}-1\geq 2$ (otherwise we would have a cycle of length $1$ in $\sigma_2$ and we could just ignore it), we get
$$
s_{1}-s_{0}<s_{1}-s_{0}+s_{2}-s_{1}-2=(s_{1}-s_{0}-1)+(s_{2}-s_{1}-1)=\ell.
$$  
This shows that $d(\sigma_2,\eta)< \ell$. Analogously we have that $d((1),\eta)< \ell$. In fact $(s_{2} \, \cdots \, s_{1}+1 \ s_{1})$ decomposes into $s_{2}-s_{1}$ transpositions, which is less than $(s_{1}-s_{0}-1)+(s_{2}-s_{1}-1)$. Using the same approach at the end of Case 1, we get the desired conclusion also in this case.

%
% which gives us a nonzero evaluation of the matrices $E$. If $\eta$ considers paths given by those permutations in $R_{i}$, then we are in the same situation as for even length. Assume then that the path is done by the permutation $(s_{2},\dots,s_{1}+1,s_{1})$. Note then that in computing $\sigma_2  \eta^{-1}$ we have $(s_{1},s_{2},s_{1}-1,s_{1}-2,\dots,s_{0}+2,s_{0}+1)$, which decomposes into $s_{1}-s_{0}$ transpositions. Note that $s_{2}-s_{1}-1\geq 2$ (otherwise we would have a cycle of length $1$ in $\sigma_2$ and we could just ignore it). Hence
%$$
%s_{1}-s_{0}<s_{1}-s_{0}+s_{2}-s_{1}-2=(s_{1}-s_{0}-1)+(s_{2}-s_{1}-1)=\ell
%$$  
%This shows that $d(\sigma_2,\eta)< \ell$. Analogously we have that $d((1),\eta)< \ell$, since $(s_{2},\dots,s_{1}+1,s_{1})$ decomposes into $s_{2}-s_{1}$ transpositions which is less than $(s_{1}-s_{0}-1)+(s_{2}-s_{1}-1)$.

\smallskip

In order to complete the proof of the theorem we are left to analyze the general situation. Recall that
$$
\sigma_2 = (s_1\,s_1-1\,\cdots\,1) \cdots (s_h\,s_h-1\,\cdots\,s_{h-1}+1),
$$
where the first $p$ cycles are of even length ($s_i = 2q_i$) and the remaining $h-p$ are of odd length, $p \in \{ 0, 1, \ldots, h \}$. Note that $h-p > 0$, otherwise we are in Case 1.  

Now distinguish two situations: $h-p$ odd or $h-p$ even.  

Suppose first that $h-p = 2k + 1$ is odd. In this case we construct the sequence $E = (G_1, \ldots, G_p, G_1', \ldots, G_k', G_{k+1}')$ in blocks as follows:
\begin{itemize}
\item[•] 
for the blocks $G_i$, $i = 1, \ldots, p$, we use the idea of Case 1. More precisely we put
\begin{equation*}
\begin{split}
G_i &= e_{q_{i-1}+1,q_{i-1}+1}, e_{q_{i-1}+1,q_{i-1}+2}, e_{q_{i-1}+2,q_{i-1}+2}, \ldots,
e_{q_{i}-1,q_{i}}, \\
& \ \ \ \  e_{q_{i},q_{i}} + e_{q_{i},q_{i}+1}, 
e_{q_{i},q_{i}+1} + e_{q_{i-1}+1,q_{i-1}+1},
\end{split}
\end{equation*}
where we assume that $q_0 = 0$;
\vspace{0.1 cm}

\item[•] for the blocks $G_j'$, $j= 1, \ldots, k$, we mimic the first block of matrices (called $E_1$) given in Case 3, more  precisely
\begin{equation*}
\begin{split}
G_j' &= e_{q_{j-1+p}+1, q_{j-1+p}+1}, e_{q_{j-1+p}+1, q_{j-1+p}+2},
\ldots, e_{q_{j+p}, q_{j+p}+1}, \\
& \ \ \ \ e_{q_{j+p}+1, q_{j+p}+1} + e_{q_{j-1+p}+1, q_{j-1+p}+1}, e_{q_{j+p}+1, q_{j+p}+2}, \\
& \ \ \ \ e_{q_{j+p}+2, q_{j+p}+2}, \ldots, e_{q_{j+p+1}-1, q_{j+p+1}}, \\
& \ \ \ \ e_{q_{j+p+1}, q_{j+p+1}} + e_{q_{j+p+1}, q_{j+p+1}+1}, e_{q_{j+p+1}, q_{j+p+1}+1} + e_{q_{j+p}+1, q_{j+p}+1};
\end{split}
\end{equation*}
\vspace{0.1 cm}

\item[•] the block $G_{k+1}'$ is constructed as in Case 2:
\begin{equation*}
\begin{split}
G_{k+1}' &= e_{q_{p+k+1}+1, q_{p+k+1}+1}, 
e_{q_{p+k+1}+1, q_{p+k+1}+2}, \ldots,
e_{q_{p+k+2}, q_{p+k+2}+1}, \\
& \ \ \ \  e_{q_{p+k+2}+1, q_{p+k+2}+1} - 
\frac{\alpha_{(1)}}{\alpha_{\sigma_2}} e_{q_{p+k+1}+1, q_{p+k+1}+1}, \\ 
& \ \ \ \ e_{q_{p+k+2}+1, q_{p+k+2}+2}, e_{q_{p+k+2}+2, q_{p+k+2}+2}, \ldots, e_{m', m_0}.
\end{split}
\end{equation*}
\end{itemize}

%
%
%
%
%\begin{itemize}
%\item[•] for the blocks $G_i$, $i= 1, \ldots, p$, we just use the corresponding $E_i$'s of Case 1 (recall that $p = h$ in Case 1);
%\vspace{0.1 cm}
%
%\item[•] for the blocks $G_j'$, $j= 1, \ldots, k$, we use the first block of matrices (called $E_1$) given in Case 3, according to the corresponding cycles;
%
%\vspace{0.1 cm}
%
%\item[•] the block $G_{k+1}'$ is constructed by using the entire sequence $E$ given in Case 2, according to the corresponding cycle.
%\end{itemize}
%
%\textcolor{red}{I think that the sequence $E$ constructed in the general case is not correct. See that when we join $G_p$, $G_1^\prime,\dots,  G_k^\prime$, in the way that we wrote, we are also joining the scalar $-\alpha_1/\alpha_{\sigma_2}$ in all this blocks. But this scalar should appear only in the last block of the sequence $E$ from the general case (right?). So, we should write something that when we join that sequence from the previous cases which has $-\alpha_1/\alpha_{\sigma_2}$, we should ignore this scalar and write just a sum (just a plus sign).}

%
%
%For the first $p$ cycles of even length we just apply Case 1.
%
%Now we have to distinguish two situations: $h-p$ odd or $h-p$ even. 
%
%Suppose first that $h-p = 2k + 1$ is odd. 
%We replace the last block of the sequence $E$ given by the first $p$ cycles by the first block of Case 3, $k$ times, and finally we join the sequence of matrices given in Case 2. 

Now assume that $h-p = 2k$ is even. In this case we do not have the last block $G_{k+1}'$. All the other ones are constructed as in the previous case, except for the last one $G_k'$, which becomes:
\begin{equation*}
\begin{split}
G_k' &= e_{q_{k-1+p}+1, q_{k-1+p}+1}, e_{q_{k-1+p}+1, q_{k-1+p}+2},
\ldots, e_{q_{k+p}, q_{k+p}+1}, \\
& \ \ \ \ e_{q_{k+p}+1, q_{k+p}+1} + e_{q_{k-1+p}+1, q_{k-1+p}+1}, e_{q_{k+p}+1, q_{k+p}+2}, \\
& \ \ \ \ e_{q_{k+p}+2, q_{k+p}+2}, \ldots, e_{q_{k+p+1}-1, q_{k+p+1}}, \\
& \ \ \ \ e_{q_{k+p+1}, q_{k+p+1}} + e_{q_{k+p+1}, q_{k+p+1}+1}, 
e_{q_{k+p+1}, q_{k+p+1}+1} - 
\frac{\alpha_{(1)}}{\alpha_{\sigma_2}} e_{q_{k+p}+1, q_{k+p}+1}, \\
& \ \ \ \ e_{q_{k+p+1}+1, q_{k+p+1}+1}, e_{q_{k+p+1}+1, q_{k+p+1}+2}, \ldots, e_{m', m_0}.
\end{split}
\end{equation*}

%
% We proceed similarly as in the previous case but with the following modifications: we do not have the last block $G_{k+1}'$ and we construct $G_k'$ by using not just the first block of matrices given in Case 3 but the whole sequence.

%
%Now we consider the general case. Here we have two situations: $h-p$ odd or $h-p$ even. Let us first deal with the case where $h-p$ is odd. For the first $p$ cycles (of even length) we just apply Case 1. If $h-p=1$, then we replace the last block of the sequence $E$ given by first $p$ cycles, by the sequence of matrices given in Case 2. If $h-p\geq 3$, then write $h-p=2k+1$, $k$ a positive integer. Then we replace the last block of the sequence $E$ given by the first $p$ cycles by the first block of Case 3, $k$ times, and finally we join the sequence of matrices given in Case 2 (notice that if $h-p$ is even, then we just do not join this last sequence of matrices given in Case 2).

In both cases, we get that the sequence $E$ is such that
$$
E_{(1)} = e_{1,m_{0}}, \ \ \ 
E_{\sigma_2} = -\dfrac{\alpha_{(1)}}{\alpha_{\sigma_2}}e_{1,m_{0}}, \ \ \
E_{\sigma} = 0 \ \mbox{for all} \ \sigma \in \supp(f)\setminus{ \{(1),\sigma_2\}}.
$$

In fact, each permutation $\eta \notin\{(1),\sigma_2\}$, giving a nonzero evaluation of the matrices in the sequence $E$, does not belong to $\supp(f)$. Indeed, in computing  $d(\eta,(1))$ and 
$d(\eta,\sigma_2)$, we always have a sum 
\[
l_1+\cdots +l_h
\]
where, for each $i$,  $l_{i}\leq s_i -s_{i-1}-1$ or $l_{i+1}+l_{i}\leq (s_{i+1}-s_{i}-1)+(s_{i}-s_{i-1}-1)$, and for at least one $i$ we have that the inequality is strict.
\end{proof}

We remark that $2d-3$ may not be the optimal bound. Finding the maximal number $N(d)$ such that $m < N(d)$ implies that $f$ is a linear combination of $g$ and a polynomial identity is left as an open question.

%\begin{example}
%If $d=2$, then the assumption in Theorem \ref{Null} reads as $m\le 3$. Proposition \ref{property of g} along with Example \ref{e5} show that the conclusion of the theorem does not hold if $m=5$. 
%The assumption is thus necessary and, in general, cannot be wakened.
%\end{example}

\end{document}